\documentclass[11pt]{amsart}

\usepackage{amssymb, amsmath, amsthm, amsfonts, amsxtra, mathtools}
\usepackage{latexsym}
\usepackage[american]{babel}
\usepackage{geometry}

\usepackage[pdfencoding=auto, psdextra, breaklinks=true]{hyperref}
\usepackage{xcolor}
\definecolor{mylinkcolor}{rgb}{0.5,0.0,0.0}
\definecolor{myurlcolor}{rgb}{0.0,0.0,0.7}
\definecolor{mycitecolor}{rgb}{0.0,0.4,0.0}
\hypersetup{
 colorlinks,
 urlcolor=myurlcolor,
 citecolor=mycitecolor,
 linkcolor=mylinkcolor,
 breaklinks=true,
 pdfauthor={Alexandre Benoist, Jean Kieffer}
}
\usepackage{url}
\usepackage[capitalize]{cleveref}

\usepackage{cite}

\usepackage{tikz}
\usepackage{tikz-cd}
\usepackage{array}
\usepackage{graphicx}

\newtheorem{theorem}{Theorem}[section]

\newtheorem{lemma}[theorem]{Lemma}
\newtheorem{prop}[theorem]{Proposition}

\theoremstyle{definition}
\newtheorem{definition}[theorem]{Definition}

\newcommand{\Q}{\mathbb{Q}}
\newcommand{\F}{\mathbb{F}}
\newcommand{\Z}{\mathbb{Z}}

\newcommand{\R}{\mathbb{R}}
\newcommand{\Zhat}{\widehat{\Z}}
\renewcommand{\O}{\mathcal{O}}
\newcommand{\frakp}{\mathfrak{p}}
\newcommand{\frakl}{\mathfrak{l}}
\renewcommand{\L}{\mathcal{L}}
\newcommand{\fraka}{\mathfrak{a}}
\newcommand{\frakb}{\mathfrak{b}}
\newcommand{\rhobar}{\overline{\rho}}
\newcommand{\rhohat}{\widehat{\rho}}
\newcommand{\tp}{\intercal}
\newcommand{\sqf}{\mathrm{sqf}}
\newcommand{\nsqf}{\mathrm{nsqf}}

\DeclareMathOperator{\embed}{\hookrightarrow}
\DeclareMathOperator{\eps}{\varepsilon}
\DeclareMathOperator{\lcm}{\mathrm{lcm}}

\DeclareMathOperator{\End}{\mathrm{End}}

\DeclareMathOperator{\Gal}{\mathrm{Gal}}
\DeclareMathOperator{\GL}{\mathrm{GL}}

\DeclareMathOperator{\GSp}{\mathrm{GSp}}
\DeclareMathOperator{\Sp}{\mathrm{Sp}}
\DeclareMathOperator{\Gm}{\mathbb{G}_m}
\DeclareMathOperator{\Tr}{\mathrm{Tr}}

\newcommand{\Split}[2]{\mathcal{S}_{#1,#2}}

\DeclareMathOperator{\Prime}{\mathcal{P}}
\DeclareMathOperator{\Diag}{\mathrm{Diag}}
\DeclareMathOperator{\id}{\mathrm{id}}

\title[The asymptotic distribution of Elkies primes is Gaussian]{The asymptotic distribution of Elkies primes for reductions of abelian varieties is Gaussian}
\author{Alexandre Benoist}
\address[Alexandre Benoist]{University of Luxembourg, Department of Mathematics. \newline
ORCID: 0009-0002-3942-0961}
\email{alexandre.benoist@uni.lu}
\author{Jean Kieffer}
\address[Jean Kieffer]{Université de Lorraine, CNRS, Inria, LORIA\newline
ORCID: 0000-0002-9953-0137}
\email{jean.kieffer@loria.fr}
\date{\today}

\begin{document}

\begin{abstract}
We generalize the notion of Elkies primes for elliptic curves to the setting of abelian varieties, possibly equipped with real multiplication (RM), and prove the following. Let~$A$ be such an abelian variety over a number field whose Galois representation has large image with respect to the chosen RM. Then the distribution of the number of Elkies primes (in a suitable range) for reductions of~$A$ modulo primes converges weakly to a Gaussian distribution around its expected value. This refines and generalizes results obtained by Shparlinski and Sutherland in the case of non-CM elliptic curves, and has implications for the complexity of the SEA point counting algorithm for abelian surfaces over finite fields.
\end{abstract}

\maketitle

\section{Introduction}

\subsection{Setup}

Let~$E$ be an elliptic curve over a finite field $\F_q$. We say that a prime number $\ell$ is \emph{Elkies} for $E$ if there exists an $\ell$-isogeny with domain $E$ defined over $\F_q$. This terminology stems from the Schoof--Elkies--Atkin (SEA) algorithm for determining $\#E(\F_q)$ \cite{schoofCountingPointsElliptic1995}; this algorithm is faster if $E$ has many small Elkies primes~$\ell$, as Elkies's method can then be applied to determine $\#E(\F_q)$ mod~$\ell$. In order to assess the overall complexity of the SEA algorithm, Shparlinski and Sutherland proved that there are enough Elkies primes on average, either when considering all elliptic curves over a fixed $\F_q$ \cite{shparlinski14} or when considering reductions of a fixed, non-CM elliptic curve over~$\Q$ modulo primes in a large interval \cite{shparlinski15}. For further results in a non-average setting, see \cite{shparlinskiProductSmallElkies2015}.

We may also consider Elkies primes for abelian varieties of higher dimensions. Let~$A$ be a polarized abelian variety of dimension $g$ over $\F_q$. We say that a prime $\ell$, coprime to~$q$ and the degree of the polarization, is Elkies for~$A$ if there exists an $\F_q$-rational subgroup $G\subset A[\ell]$ which is maximal isotropic for the Weil pairing; in that case, the quotient~$A/G$ is also equipped with a polarization of the same degree. More generally, if~$A$ has real multiplication (RM) by an order~$\O$ in a totally real number field~$K$ of degree~$d$, i.e.~if~$A$ is equipped with a primitive embedding $\O\embed\End_{\F_q}(A)$ such that every $x\in \O$ is invariant under the Rosati involution, we say that a prime ideal $\frakl$ of $\O$ is Elkies for $A$ if $A[\frakl]$ admits a maximal isotropic subgroup~$G$ defined over~$\F_q$ and stable under~$\O$, or in other words, if there exists an $\F_q$-rational $\mathfrak l$-isogeny from~$A$, as defined in~\cite{brooksIsogenyGraphsOrdinary2017}. This notion of Elkies primes is a suitable analogue of the classical definition in the context of the SEA algorithm on principally polarized abelian surfaces with or without RM~\cite{kiefferCountingPointsAbelian2022}.

\subsection{Main results}

In this paper, we show that the number of Elkies primes in certain ranges for reductions of a fixed abelian variety~$A$ with RM over a number field asymptotically follows a Gaussian distribution, provided that the Galois representation attached to~$A$ has a large enough adelic image.

To formulate this last condition precisely, we introduce the following notation. Let~$F$ be the field of definition of $A$ and let~$G_F$ be its absolute Galois group. If~$\ell$ is a large enough prime, the $\ell$-adic Tate module~$T_\ell(A)$ of~$A$ is a free $\O\otimes \Z_\ell$-module of rank $2h$ where $h = g/d$, endowed with an nondegenerate alternating form with values in $\O\otimes \Z_\ell$, as we review in~\cref{sec:galrep}.
If~$n$ is a sufficiently large integer, we can therefore consider the global Galois representation
\[
\rhohat_{n}: G_F \to \GSp_{2h}(\O\otimes \Zhat_{\geq n}), \quad \text{where }\ \Zhat_{\geq n} := \prod_{\ell \text{ prime, } \ell\geq n} \Z_\ell.
\]
We say that $A$ has \emph{large Galois image} if $\rhohat_{n}(G_F)$ contains~$\Sp_{2h}(\O\otimes \Zhat_{\geq n})$ for large enough~$n$. Assuming that~$\O$ is the whole endomorphism ring of~$A$ over~$\overline{\Q}$ (a necessary condition), one can sometimes guarantee that~$A$ has large Galois image, as in Serre's open image theorem in the case $d=1$ \cite{serreResumeCoursAu1985}: we review this theorem and its RM analogues in~\cref{sec:galrep}.

Our main result on the distribution of Elkies primes is then the following.

\begin{theorem}
\label{thm:main-cv}
    Assume the generalized Riemann hypothesis (GRH). Let~$\O$ be an order in a totally real number field~$K$ of degree~$d$, and let~$A$ be a polarized abelian variety of dimension~$g\geq 1$ defined over a number field~$F$ with RM by~$\O$ with large Galois image.
    
    For a real number~$L$, denote by $\Prime_K(L,2L)$ the set of prime ideals $\frakl$ of $K$ such that $N_{K/\Q}(\frakl) \in [L,2L]$, and define $\Prime_F(P,2P)$ similarly. For a prime $\frakp$ of $F$ of good reduction for~$A$ and $L\geq 1$, let $N_e(\frakp, L)$ be the number of Elkies primes $\frakl \in \mathcal{P}_K(L,2L)$ for $A_{\frakp}$. Further define $\alpha_h\in (0,1)$ by the formula
    \[
        \alpha_h = \sum\limits_{(d_1,\ldots,d_r) \in \Sigma_h} \frac{1}{2^r} \cdot \prod\limits_{i=1}^r \frac{1}{d_i} \cdot \prod\limits_{k=1}^h \frac{1}{\# \{j \; : \; d_j = k \} ! }  
    \] where $\Sigma_h$ denotes the set of unordered partitions of the integer $h = g/d$.

    Then, as $L,P\to \infty$ with $P \gg L^n$ for every positive integer $n$, the function
    \[
        \begin{matrix}
            X_{P,L}: & \Prime_F(P,2P) &\longrightarrow &\R \\
            & \frakp & \longmapsto & \displaystyle\frac{N_e(\frakp,L)- \alpha_h \# \mathcal{P}_K(L,2L)}{\sqrt{\alpha_h (1-\alpha_h)\# \mathcal{P}_K(L,2L)}}
        \end{matrix}
    \]
    converges weakly to the standard Gaussian distribution with mean value $0$ and variance $1$.
\end{theorem}

Intuitively,~$\alpha_h$ is the probability that~$\frakl$ will be Elkies for~$A_\frakp$ for random~$\frakl$ and~$\frakp$; weak convergence to the Gaussian distribution of \cref{thm:main-cv} is what we would obtain from the central limit theorem in the naive probabilistic model where the events ``$\frakl$ is Elkies for $A_\frakp$" are all independent. We list the first few values of $\alpha_h$ in~\cref{tab:alpha_g}.

\begin{table}[ht]
\renewcommand{\arraystretch}{1.2}    \centering
    \begin{tabular}{|c|c|c|c|c|c|c|c|c|c|c|c|c|} \hline
        $h$ & 1 & 2 & 3 & 4 & 5 & 6 & 7 & 8  \\ \hline
        $\alpha_h$ (exact value) & $\frac{1}{2}$ & $\frac{3}{8}$ & $\frac{5}{16}$ & $\frac{35}{128}$ & $\frac{63}{256}$ & $\frac{231}{1024}$ & $\frac{429}{2048}$ & $\frac{6435}{32768}$  \\ \hline
        $ \alpha_h$ (approximate value) & 0.5 & 0.375 & 0.3125 & 0.2734 & 0.2461 & 0.2256 & 0.2095 & 0.1964 \\ \hline
    \end{tabular}
    \smallskip
    \caption{Values of $\alpha_h$}
    \label{tab:alpha_g}
\end{table}

We prove \cref{thm:main-cv} by analyzing the moments of~$X_{P,L}$ at all orders $k\geq 0$, which we somewhat abusively denote by
\begin{displaymath}
    \mathbb{E}(X_{P,L}^k) := \frac{1}{\# \mathcal{P}_F(P,2P)} \sum_{\frakp\in \mathcal{P}_F(P,2P)} X_{P,L}(\frakp)^k. 
\end{displaymath}
In fact, \cref{thm:main-cv} follows directly from the following result: see \cite[Theorem 30.2]{billingsley95}.

\begin{theorem}
\label{thm:main-moments}
    Assume GRH, and keep notation from \cref{thm:main-cv}. Let~$k\geq 0$ be any integer, and let $M_k$ be the moment of order $k$ of the standard Gaussian distribution (thus $M_k=0$ for odd $k$). Then $\mathbb{E}(X_{P,L}^k)$ converges to $M_k$ as $P,L \to \infty$ with $P \gg L^n$ for every positive integer $n$. More precisely, we have
    \[
     \mathbb{E}(X_{P,L}^k) = M_k + O_{A,k}\left( \frac{1}{L^{1/2}\log(L)^{1/2}} + \frac{L^{k(2h^2 + h + 3/2)}\log(P)^2}{\log(L)^{k/2}P^{1/2}}\right).
    \]
\end{theorem} 
Here the notation $O_{A,k}$ means that the implicit constants in Landau's notation are allowed to depend on $A$ (hence on $F$, $\O$, and~$h$) and~$k$.

In the case of elliptic curves, \cref{thm:main-moments} refines \cite{shparlinski15} as we consider moments of all orders and provide an asymptotic equivalent of the even moments rather than an upper bound. To the best of our knowledge, \cref{thm:main-moments} is also the first quantitative result on the distribution of Elkies primes in higher dimensions. In particular, a consequence of this theorem is that there are enough Elkies primes to run the SEA algorithm in dimension~2 on average over reductions of a fixed abelian variety: see \cite[Def.~3.7]{kiefferCountingPointsAbelian2022}.

The proof of \cref{thm:main-moments} is inspired from~\cite{shparlinski15}: we apply an explicit version of the \v{C}ebotarev density theorem (which relies on GRH) to number field extensions of~$F$ cut out by torsion subgroups of $A$, and count how many elements in their Galois groups correspond to $\frakl$ being Elkies for $A_{\frakp}$. The result then follows from rearranging the summations and from a combinatorial argument to determine the leading term in the moments of $X_{P,L}$.

We also provide numerical experiments on the distribution of Elkies primes in large ranges in the case of elliptic curves: it was actually the very smooth aspect of the data which prompted us to try and prove Theorem \ref{thm:main-cv}.

One might wonder if this convergence result to a Gaussian distribution also holds when considering all elliptic curves (or more generally abelian varieties) over a fixed~$\F_q$ as in \cite{shparlinski14}. To answer this, it seems that one would need careful control on the class numbers appearing in the distribution of traces of Frobenius for elliptic curves over~$\F_q$.

\subsection{Organization} In~\cref{sec:galrep}, we review the properties of Galois representations attached to abelian varieties with RM, characterize Elkies primes both in terms of Frobenius elements in the Galois representation and in terms of the existence of isogenies, and recall results from the literature on large Galois images. In~\cref{sec:gsp}, we count matrices in $\GSp_{2h}(\O/\frakl\O)$ (and related groups) corresponding to Elkies primes, a key input to the \v{C}ebotarev density theorem. We prove \cref{thm:main-moments} in~\cref{sec:distribution}, and report on our numerical experiments in~\cref{sec:experiments}.

\subsection{Acknowledgments} The first author was supported by the Agence Nationale de la Recherche/France 2030 grant CRYPTANALYSE (reference 22-PECY-0010.) This research was funded in whole, or in part, by the Luxembourg National Research Fund (FNR), grant reference PRIDE23/18685085. For the purpose of open access, and in fulfilment of the obligations arising from the grant agreement, the authors have applied a Creative Commons Attribution 4.0 International (CC BY 4.0) license to any Author Accepted Manuscript version arising from this submission.

\subsection{Statement on competing interests} The authors declare no competing interests.

\subsection{Data availability statement}
The code used to generate the figures in \cref{sec:experiments} is available as
one of the paper's source files at \url{https://arxiv.org/abs/2411.18171}.

\section{Galois representations and Elkies primes}
\label{sec:galrep}

In this section, we review basic facts on the structure of torsion subgroups of abelian varieties with RM over any field (§\ref{subsec:torsion}). Then we characterize Elkies primes for such abelian varieties in terms of the existence of isogenies (§\ref{subsec:elkies-isog}) and, in the case of finite fields or reductions of abelian varieties over number fields, in terms of the action of Frobenius on torsion subgroups~(§\ref{subsec:elkies-frob}).
Finally, we review deeper results on large Galois images (§\ref{subsec:openimage}).

\subsection{Torsion subgroups of abelian varieties with RM}
\label{subsec:torsion}

Throughout, we use the notation listed in \cref{tab:notations}. For the reader's convenience, the table also includes symbols defined later in this section. For now, $F$ is any field, and~$A$ is a polarized abelian variety over~$F$ with real multiplication by an order~$\O$ as in the introduction. We write~$d_A$ for the degree of the polarization of~$A$ and~$c_\O$ for the conductor of~$\O.$

\begin{table}
\begin{tabular}{lp{0.81\textwidth}}
    $K$ & a totally real number field \\
    $d$ & the degree of~$K$ over~$\Q$ \\
    $\O$ & an order of~$K$ \\
    $\O_K$ & the ring of integers in~$K$ \\
    $c_\O$ & the conductor of~$\O$, an ideal supported at primes dividing~$[\O_K:\O]$ \\
    $N_{K/\Q}$ & the norm map for ideals or elements of~$K/\Q$ \\
    $\Tr_{K/\Q}$ & the trace map for elements of~$K/\Q$ \\
    $\ell$ & a prime number in~$\Z_{\geq 1}$ \\
    $\frakl$ & a prime ideal of~$\O$ above~$\ell$ \\
    \\
    $F$ & a field \\
    $p$ & the characteristic exponent of~$F$ (a prime number, or~$1$ if $\mathrm{char}(F)=0$) \\
    $G_F$ & the absolute Galois group of~$F$ \\
    $\chi_\ell$ & the cyclotomic character $G_F\to \Z_\ell^\times$ \\
    \\
    $A$ & a polarized abelian variety over~$F$ with RM by~$\O$, i.e.~endowed with a primitive embedding $\O\embed \End_F(A)$, with $1 \mapsto \id_A$, whose image consists of elements that are invariant under the Rosati involution\\
    $g$ & the dimension of~$A$; in particular $d|g$ \\
    $h$ & the integer $g/d$ \\
    $d_A$ & the degree of the polarization of~$A$ \\
    $\pi_A$ & the Frobenius endomorphism of~$A$ (if~$F$ is finite) \\
    $T_\ell(A)$ & the $\ell$-adic Tate module of~$A$ \\
    $e_\ell$ & the Weil pairing on~$T_\ell(A)$, with values in~$\Z_\ell$ \\
    $A[\ell]$ & the $\ell$-torsion subgroup of~$A$ \\
    $\psi_\ell$ & the~$\O$-linear alternating form on~$A[\ell]$ defined in \cref{lem:basic-galrep} \\
    $A[\frakl]$ & the $\frakl$-torsion subgroup of~$A$, as defined in~\eqref{eq:frakl-torsion} below \\
    $\rho_\ell$ & the $\ell$-adic Galois representation with target~$\GSp_{2h}(\O\otimes\Z_\ell)$, cf.~\eqref{eq:galrep-ell} below\\
    $\rhobar_\ell$, $\rhobar_\frakl$ & the Galois representations modulo~$\ell$ and~$\frakl$ as in~\eqref{eq:galrep-mod-ell}, \eqref{eq:galrep-mod-frakl} below \\
    \\
    $\lambda$ & the multiplier character $\GSp_{2h} \to \Gm$, as in \cref{def:symplectic} \\
    $\GSp_{2h}(R;U)$ & the subset of $\GSp_{2h}(R)$ given by $\lambda^{-1}(U)$, as in \cref{def:symplectic} \\
    $\Split{2h}{\F_q}(\lambda_0)$ & the split matrices in $\GSp_{2h}(\F_q)$ with multiplier~$\lambda_0$, as in \cref{def:split}.
\end{tabular}
\bigskip
\caption{List of notations}
\label{tab:notations}
\vspace{-1em}
\end{table}

Recall that whenever~$n\geq 1$ is prime to~$p$, the $n$-torsion subgroup $A[n]$ of~$A$, seen as group scheme, is étale. Throughout, we abuse notation and identify these group schemes (as well as their subgroups) with their groups of points over a separable closure~$F^{sep}$ of~$F$, endowed with an action of the absolute Galois group~$G_F$ of~$F$. For all such~$n$, there is a canonical nondegenerate pairing $e_n$ on $A[n] \times A^\vee[n]$, called the Weil pairing, whose values are $n$-th roots of unity in~$F^{sep}$. Now if $\ell\neq p$ is a prime number, by making compatible choices of $\ell$-power roots of unity in~$F^{sep}$ and by composing with the polarization of~$A$ on the second argument, we obtain a new pairing
\[
    e_\ell: T_\ell(A)\times T_\ell(A)\to \Z_\ell
\]
on the $\ell$-adic Tate module~$T_\ell(A)$. We will work with this version of the Weil pairing in the rest of the paper. The Tate module~$T_\ell(A)$ is then a free~$\Z_\ell$-module of rank~$2g$ on which~$e_\ell$ is nondegenerate. 
If further~$\ell$ is prime to~$d_A$, then $e_\ell$ also gives a nondegenerate alternating form on~$A[\ell]$ with values in~$\F_\ell = \Z/\ell\Z$. We may also view~$T_\ell(A)$ as an $\O\otimes\Z_\ell$-module, using the action of~$\O$ as endomorphisms of~$A$.

\begin{lemma}
\label{lem:basic-galrep}
Assume that~$\ell$ is coprime to~$p$, $d_A$ and~$c_\O$, so that $\O\otimes\Z_\ell = \O_K\otimes\Z_\ell$.
\begin{enumerate}
     \item \label{it:basic-galrep-free} $T_\ell(A)$ is a free $\O\otimes \Z_\ell$-module of rank $2h$.
     \item \label{it:basic-galrep-psi} There exists a unique $\O\otimes \Z_\ell$-bilinear alternating form $\psi_\ell: T_\ell(A)\times T_\ell(A) \to \O\otimes \Z_\ell$ with the following property: for every $x,y\in T_\ell(A)$, $e_\ell(x,y) = \Tr_{K/\Q}(\psi_\ell(x,y))$.
     \end{enumerate}
\end{lemma}

\begin{proof}
\begin{enumerate}
    \item This is \cite[Prop.~2.1.1]{ribetGaloisActionDivision1976}.
    \item The existence and uniqueness of~$\psi_\ell$ after tensoring with~$\Q_\ell$ is \cite[Lemma~1.2.1]{chi$ell$adic$lambda$adicRepresentations1992}. In fact, $\psi_\ell$ exists at the level of $\O\otimes\Z_\ell$-modules by \cite[Lemma 3.1]{banaszakImage$ell$adicGalois2006}. \qedhere
\end{enumerate}
\end{proof}

Under the assumptions of \cref{lem:basic-galrep}, we also consider the decomposition of $\O/\ell\O$ as a product of fields:
\[
     \O/\ell\O = \prod_{\frakl|\ell} \O/\frakl\O.
\]
Then for each $\frakl|\ell$, we define the $\frakl$-torsion subgroup~$A[\frakl]\subset A[\ell]$ as
\begin{equation}
\label{eq:frakl-torsion}
    A[\frakl] = \bigcap_{f\in \frakl} \ker(f) = \{x\in A[\ell]: f(x) = 0 \text{ for every } f\in \frakl\}.
\end{equation}

\begin{lemma}
    \label{lem:tors-direct-sum}
    Assume that~$\ell$ is coprime to~$p$, $d_A$ and $c_\O$. Then we have a direct sum decomposition
    \[
    A[\ell] = \bigoplus_{\frakl|\ell} A[\frakl]
    \]
    where for each~$\frakl|\ell$, the summand~$A[\frakl]$ is an $(\O/\frakl\O)$-vector space of dimension~$2h$. This direct sum is orthogonal with respect to~$\psi_\ell$, and the restriction of~$\psi_\ell$ to each~$A[\frakl]$ is nondegenerate.
\end{lemma}

\begin{proof}
    The decomposition of~$A[\ell]$ as a direct sum is a consequence of~\cref{lem:basic-galrep}\eqref{it:basic-galrep-free}. Let us check that this decomposition is orthogonal with respect to~$\psi_\ell$. Let~$\frakl\neq\frakl'$ be prime ideals above~$\ell$, and fix an element~$f\in \frakl$ which is invertible modulo~$\frakl'$.
    If~$x\in A[\frakl]$ and $y\in A[\frakl']$, then there exists $y'\in A[\frakl']$ such that $y = f(y')$. By~$\O$-linearity of~$\psi_\ell$, we get
    \[
        \psi_\ell(x,y) = \psi_\ell(x, f(y')) = \psi_\ell(f(x), y') = \psi_\ell(0,y') = 0.
    \]
    Finally, each piece is nondegenerate by \cite[Lemma 3.2]{banaszakImage$ell$adicGalois2006}.
\end{proof}

We now include the action of the Galois group~$G_F$ in the picture. Let~$\ell$ be coprime to~$p$, $d_A$ and $c_\O$. By equivariance of the Weil pairing (see for instance \cite[Lemma 4.7]{banaszakImage$ell$adicGalois2006}), we have for all $\sigma\in G_F$ and $x,y\in T_\ell(A)$:
\[
    e_\ell(\sigma(x),\sigma(y)) = \chi_\ell(\sigma) e_\ell(x,y).
\]
The action of~$\sigma$ on~$A[\ell]$ is also $\O$-linear because the elements of~$\O$, seen as endomorphisms, are defined over~$F$ by assumption. By nondegeneracy of~$\Tr_{K/\Q}$, we have for all $x,y\in T_\ell(A)$:
\[
    \psi_\ell(\sigma(x), \sigma(y)) = \chi_\ell(\sigma) \psi_\ell(x,y).
\]
In other words,~$\sigma$ preserves $\psi_\ell$ up to multiplication by the scalar $\chi_\ell(\sigma)\in \Z_\ell^\times$.

In order to identify the action of~$\sigma$ on~$A[\ell]$ as an element in a standard symplectic group, we choose once and for all a symplectic basis $(v_1,\ldots,v_{2h})$ of~$T_\ell(A)$ as an $\O\otimes \Z_\ell$-module. This means that the alternating form~$\psi_\ell$ in this basis takes the standard form
\[
J_{2h} = \begin{pmatrix} 0 & I_h \\ -I_h & 0 \end{pmatrix},
\]
where~$I_h$ denotes the $h\times h$ identity matrix. We summarize our notation for the attached symplectic group in the following definition.

\begin{definition}
\label{def:symplectic}
We denote by~$\GSp_{2h}$ the general symplectic group with respect to the standard form~$J_{2h}$: for any commutative ring~$R$, we have
\[
    \GSp_{2h}(R) = \{m\in \GL_{2h}(R): m^\tp J_{2h}\, m = \lambda(m) J_{2h} \text{ for some } \lambda(m)\in R^\times \}.
\]
We call the character $\lambda: \GSp \to \Gm$ appearing in this equation the \emph{multiplier}. The kernel of~$\lambda$ is~$\Sp_{2h}$, the usual symplectic group.
If~$U$ is a subset of~$R^\times$, we also write
\[
    \GSp_{2h}(R; U) = \{m\in \GSp_{2h}(R): \lambda(m)\in U\}.
\]  
\end{definition}

Assuming that~$\ell$ is coprime to~$d_A$ and~$c_\O$, the same vectors $v_1,\ldots,v_{2h}$ also form a symplectic basis of~$A[\ell]$ as an $\O/\ell\O$-module, and for every prime~$\frakl|\ell$ of~$\O$, a symplectic basis of~$A[\frakl]$ as an $(\O/\frakl\O)$-vector space.

Summarizing, we identify the action of~$\sigma\in G_F$ on $T_\ell(A)$ with an element of the general symplectic group,
\begin{equation}
\label{eq:galrep-ell}
    \rho_\ell(\sigma) \in \GSp_{2h}(\O \otimes \Z_\ell; \Z_\ell^\times) \subset \GSp_{2h}(\O \otimes \Z_\ell)
\end{equation}
such that
\[
    \lambda(\rho_\ell(\sigma)) = \chi_\ell(\sigma) \in \Z_\ell^\times \subset (\O\otimes\Z_\ell)^\times.
\]
We identify the action of~$\sigma$ on the $\ell$-torsion subgroup $A[\ell]$ with the image of~$\sigma$ under the reduced representation
\begin{equation}
    \label{eq:galrep-mod-ell}
    \rhobar_\ell(\sigma)\in \GSp_{2h}(\O/\ell\O; \F_\ell^\times).
\end{equation}
For each prime $\frakl$ of~$\O$, we also identify the action of $\sigma$ on~$A[\frakl]$ with an element
\begin{equation}
\label{eq:galrep-mod-frakl}
    \rhobar_\frakl(\sigma) \in \GSp_{2h}(\O/\frakl\O).
\end{equation}
with the same multiplier $\chi_\ell(\sigma)$. We call $\rho_\ell$ the \emph{$\ell$-adic Galois representation}, and $\rhobar_\ell$ (resp.~$\rhobar_\frakl$) the \emph{Galois representation modulo $\ell$} (resp~$\frakl$), attached to~$A$.
By \cref{lem:basic-galrep}\eqref{it:basic-galrep-free} and \cref{lem:tors-direct-sum}, the decompositions
\[
    A[\ell] = \bigoplus_{\frakl|\ell} A[\frakl] \quad\text{and}\quad \GSp_{2h}(\O/\ell\O) = \prod_{\frakl|\ell} \GSp_{2h}(\O/\frakl\O)
\]
are compatible in the sense that the following diagram commutes:
\begin{center}
\begin{tikzcd}
    G_F \ar[r, "{\rhobar}_\ell"] \ar[rr, swap, bend right=15, "{\rhobar}_\frakl"] & \GSp_{2h}(\O/\ell\O; \F_\ell^\times) \ar[r] & \GSp_{2h}(\O/\frakl\O).
\end{tikzcd}    
\end{center}
In particular, if $\ell$ splits completely in~$\O$, then $\rhobar_\ell(G_F)$ is a subgroup of $\GSp_{2h}(\F_\ell)^d$ consisting of tuples of matrices $(m_1,\ldots,m_d)$ such that $\lambda(m_1)=\cdots=\lambda(m_d)$. At the other extreme, if~$\ell$ is inert in~$\O$, then $\rhobar_\ell(G_F)$ is a subgroup of $\GSp_{2h}(\F_{\ell^d}; \F_\ell^\times)$.

The representation $\rhobar_\frakl$ can be seen as the restriction modulo $\frakl$ of the $\frakl$-adic representation considered in~\cite[§1.1]{chi$ell$adic$lambda$adicRepresentations1992}.

\subsection{Elkies primes for abelian varieties with RM}
\label{subsec:elkies-isog}

Let us restate the definition of Elkies primes given in the introduction. We are mainly interested in finite fields, but for now, our discussion remains valid over any field $F$. We keep notation from~\cref{tab:notations}, and assume throughout that the prime ideals~$\frakl$ we consider are coprime with~$p$, $d_A$, and~$c_\O$.

\begin{definition}
    \label{def:Elkies}
    We say that $\frakl$ is \emph{Elkies} for $A$ if there exists an $F$-rational subgroup of $A[\frakl]$ that is maximal isotropic for the Weil pairing~$e_\ell$ and stable under~$\O$. Note that this last condition is automatic when $N_{F/\Q}(\frakl) = \ell$, as $\O/\frakl\O$ only consists of scalars.
\end{definition}

We can equivalently phrase this definition in terms of isotropic subspaces for~$\psi_\ell$.

\begin{lemma}
\label{lem:elkies-isotropic} The prime $\frakl$ is Elkies for $A$ if and only if there exists a maximal isotropic sub-$(\O/\frakl\O)$-vector space of $A[\frakl]$ that is maximal isotropic for~$\psi_\ell$ and~$F$-rational.
\end{lemma}

\begin{proof}
    Suppose~$\frakl$ is Elkies for~$A$, i.e.~there exists an $F$-rational $\F_\ell$-vector space $G\subset A[\frakl]$ which is maximal isotropic for the Weil pairing and stable under~$\O$. We may also view $G\subset A[\frakl]$ as an $F$-rational sub-$\O/\ell\O$-vector space of dimension~$h$. By \cref{lem:basic-galrep}\eqref{it:basic-galrep-psi}, the trace of $\psi_\ell$ vanishes on~$G\times G$, so~$\psi_\ell$ vanishes on~$G\times G$ as well as the trace is nondegenerate.

    Conversely, if~$G\subset A[\frakl]$ be a maximal isotropic subspace for~$\psi_\ell$ in~$A[\frakl]$. Seen as an~$\F_\ell$-vector space,~$G$ is isotropic for the Weil pairing by \cref{lem:basic-galrep}\eqref{it:basic-galrep-psi}, and is maximal for dimension reasons. Therefore, $\frakl$ is Elkies for~$A$.
\end{proof}

\Cref{def:Elkies} is a suitable generalization of the notion of Elkies primes for elliptic curves~\cite{shparlinski14}, abelian surfaces without RM~\cite[§3.2]{kiefferCountingPointsAbelian2022}, and abelian surfaces with RM in the case of split primes~\cite[§4.1]{kiefferCountingPointsAbelian2022}. Moreover, there is still a close link between Elkies primes and the existence of $F$-rational isogenies compatible with the RM structure and the polarization of~$A$. Let us specify this link in more detail, for motivation only, as it will not be used in the rest of the paper.

First we introduce the following notation. The Néron--Severi group $\mathrm{NS}(A)$ of~$A$ (the group of line bundles on~$A$ up to algebraic equivalence) is related to the endomorphisms of~$A$, as follows. The $\Q$-algebra $\End^0(A) = \End_{\overline{F}}(A)\otimes\Q$ is endowed with the Rosati involution~$\dagger$ coming from our choice of polarization on~$A$. Let~$\End^0(A)^\dagger$ denote the sub-vector space of elements invariant under~$\dagger$, and $\End(A)^\dagger = \End^0(A)^\dagger\cap \End(A)$. There is an isomorphism $\mathrm{NS}(A)\otimes\Q\simeq \End^0(A)^\dagger$, which depends on the chosen polarization of~$A$ \cite[(3) p.\,190]{mumford70}. Given $\alpha\in \End(A)^\dagger$ and two polarized abelian varieties $A,B$ with RM by~$\O$, we say that an isogeny $\phi: A\to B$ is an \emph{$\alpha$-isogeny} if the RM structures of~$A$ and~$B$ are compatible via~$\phi$, and if the pullback of the polarization of~$B$ via~$\phi$ (seen as an element of $\mathrm{NS}(A)$) corresponds to~$\alpha$ via the previous isomorphism. The element~$\alpha$ is then necessarily totally positive \cite[(IV) p.\,209]{mumford70}. Equivalently, we ask that the diagram
\begin{center}
\begin{tikzcd}
    A \ar[r, "\alpha"] \ar[d, "\phi"] & A \ar[r] & A^\vee \\ B \ar[rr] && B^\vee \ar[u, swap, "\phi^{\vee}"]
\end{tikzcd}
\end{center}
commutes, where $\vee$ denotes duals and the unlabeled arrows are the polarizations. This implies that~$\ker(\phi)$ is maximal isotropic in~$A[\alpha]$ for its canonical nondegenerate pairing; conversely, if $G\subset A[\alpha]$ is a maximal isotropic subspace, then $A/G$ carries a unique polarization of degree $d_A$ such that the quotient isogeny $A\to A/G$ is an $\alpha$-isogeny~\cite[Cor.~p.\,231]{mumford70}. Recall that our Elkies primes are prime to~$p$, $d_A$ and $c_\O$.

\begin{prop}
\label{prop:elkies-isog}
    \begin{enumerate}
        \item \label{it:elkies-isog-1} The prime~$\frakl$ is Elkies if and only if there exists an abelian variety~$B$ over~$F$ with RM by~$\O$ and an $F$-rational $\frakl$-isogeny $\phi: A\to B$ in the sense of \cite[Def.~4.1]{brooksIsogenyGraphsOrdinary2017}.
        \item \label{it:elkies-isog-2} Let~$\frakl_1,\ldots,\frakl_r$ be distinct Elkies primes for~$A$, and let $k_1,\ldots,k_r\geq 0$ be integers such that $\frakl_1^{k_1}\cdots\frakl_r^{k_r}$ is trivial in the narrow class group of~$\O$. Let~$\alpha\in \O$ be a totally positive generator of this product. Then there exists an abelian variety $B$ over~$F$ with RM by~$\O$ and endowed with a polarization of degree~$d_A$, and an  $F$-rational $\alpha$-isogeny $\phi:A\to B$.
    \end{enumerate}
\end{prop}

\begin{proof}[Proof of \cref{prop:elkies-isog}] \eqref{it:elkies-isog-1} directly comes from the definition of $\frakl$-isogenies.

We now prove~\eqref{it:elkies-isog-2}. For $1\leq i\leq r$, let $K_i\subset A[\frakl_i]$ be $F$-rational, maximal isotropic, and $\O$-stable subgroups as in \cref{def:Elkies}. Define now $K_i' = A[\frakl_i^{m_i}]$ if $k_i = 2m_i$ is even, and $K_i' = A[\frakl_i^{m_i + 1}] \cap \eta^{-1}(K_i)$, where~$\eta\in \O$ is any element whose $\frakl_i$-adic valuation is exactly~$m_i$, when $k_i = 2m_i+1$ is odd. We can check that $K_i'$ is independent of the choice of~$\eta$, and that it is an $F$-rational, $\O$-stable, maximal isotropic subspace in $A[\frakl_i^{k_i}]$. By \cref{lem:basic-galrep} and the Chinese remainder theorem, we have
\[A[\alpha] = \bigoplus_{i=1}^r A[\frakl_i^{k_i}].\]
Moreover, the restriction of the pairing on $A[\alpha]$ to each subgroup $A[\frakl_i^{k_i}]$ is precisely the Weil pairing $e_{\ell_i}$ (mod $\ell_i^{k_i}$), where~$\ell_i\in \Z$ denotes the prime below $\frakl_i$, and the direct sum is orthogonal, as can be seen from the functorial properties of those pairings~\cite[p.\,228]{mumford70}.
Therefore $K = K_1'\oplus\cdots\oplus K_r'$ is maximal isotropic in~$A[\alpha]$, and is the kernel of the isogeny~$\phi$ we are looking for.
\end{proof}

\subsection{Elkies primes and the action of Frobenius}
\label{subsec:elkies-frob}

We keep the notation of~\cref{tab:notations}, and assume first that \textbf{$F = \F_q$ is a finite field.} Let~$\pi_A$ denote the Frobenius endomorphism of~$A$. We continue to assume that $\frakl$ is prime to~$p$, $d_A$ and $c_\O$. We can also view the Frobenius map as an element $\pi\in G_F$.

\begin{lemma}
\label{lem:elkies-frob}
    Let~$A$ be an abelian variety over~$F = \F_q$ with RM by~$\O$. The prime $\frakl$ is Elkies for~$A$ if and only if $A[\frakl]$ admits a maximal isotropic $\O$-stable subspace stable under~$\pi_A$, if and only if $\rhobar_\frakl(\pi)\in \GSp_{2h}(\O/\frakl\O)$ stabilizes a maximal isotropic stable subspace in $(\O/\frakl\O)^{2h}$.
\end{lemma}

\begin{proof}
    This is a restatement of \cref{lem:elkies-isotropic}, using the fact that a subspace of~$A[\frakl]$ is $\F_q$-rational if and only if it is stable under~$\pi_A$.
\end{proof}

\Cref{lem:elkies-frob} prompts us to make the following definition.

\begin{definition}
\label{def:split}
    Let~$k$ be a finite field. We say that a matrix $m\in \GSp_{2h}(k)$ is \emph{split} if it leaves some maximal isotropic subspace of $k^{2h}$ stable. We denote by $\Split{2h}{k}\subset \GSp_{2h}(k)$ the subset of split matrices, and for $\lambda_0\in {k}^\times$, we write 
    \[\Split{2h}{k}(\lambda_0) := \{ m\in \Split{2h}{k}:\lambda(m) = \lambda_0\}. \]
    We note that $\Split{2h}{k}(\lambda_0)$ is a conjugacy-invariant subset of $\GSp_{2h}(k)$.
\end{definition}

Since~$\chi_\ell(\pi) = q$, another restatement of \cref{lem:elkies-isotropic} is the following.

\begin{lemma}
\label{lem:elkies-iff-frob-split}
    The prime~$\frakl$ is Elkies for~$A$ if and only if $\rhobar_\frakl(\pi) \in \Split{2h}{\O/\frakl\O}(q)$.
\end{lemma}

We now switch gears and assume that~\textbf{$F$ is a number field}. We fix a polarized abelian variety $A$ over~$F$ with RM by~$\O$. For every prime $\frakp$ of~$F$ with residue field~$F_\frakp$ of good reduction for~$A$, the reduction~$A_\frakp$ of~$A$ modulo~$\frakp$ is a polarized abelian variety of dimension~$g$ over~$F_\frakp$ with RM by~$\O$. Indeed, the listed properties can all be formulated in terms of isogenies between abelian varieties and their duals, and such isogenies extend uniquely to Néron models at~$\frakp$ by \cite[§1.4, Prop.~4]{boschNeronModels1990}. We can characterize Elkies primes for~$A_\frakp$ in terms of the Galois representations $\rhobar_\frakl$ evaluated at Frobenius elements in~$G_F$.

\begin{prop}
\label{prop:elkies-frob-nf}
    Let~$\frakp$ be a prime of good reduction for~$A$ above $p\in \Z$, and let~$\frakl$ be a prime of~$\O$ that is coprime to~$p, d_A$ and $c_\O$. Then $\frakl$ is Elkies for the reduction~$A_\frakp$ if and only if $\rhobar_\frakl(\sigma_\frakp) \in \Split{2h}{\O/\frakl\O}(N_{F/\Q}(\frakp))$, where~$\sigma_\frakp\in G_F$ is any Frobenius element at~$\frakp$ (unique up to conjugation in~$G_F$).
\end{prop}

\begin{proof}
    Denote by $F'$ the field of definition of~$A[\frakl]$, i.e.~the smallest number field such that the representation $\rhobar_\frakl:G_F\to \GSp_{2h}(\O/\frakl\O)$ factors through $\Gal(F'/F)$. Let~$\mathfrak{P}$ be a prime of $F'$ above~$\frakp$, and let~$\sigma_\frakp\in G_F$ be a Frobenius element stabilizing~$\mathfrak{P}$; we can consider~$\sigma_\frakp$ as a (uniquely specified) element of $\Gal(F'/F)$. Reduction modulo~$\mathfrak P$ defines a bijection $A[\frakl]\to A_\frakp[\frakl]$ by \cite[§1, Lemma 2]{serreGoodReductionAbelian1968}, so our choice of fixed symplectic basis of~$T_\ell(A)$ also fixes a symplectic basis of~$A_\frakp[\frakl]$ as an $(\O/\frakl\O)$-vector space. By definition, $\sigma_\frakp$ induces the Frobenius map of the extension of residue fields~$F'_{\mathfrak P} / F_\frakp$. Therefore, $\rhobar_\frakl(\sigma_\frakp)$ is precisely the matrix of the Frobenius endomorphism~$\pi_{A_\frakp}$ in the symplectic basis of~$A_\frakp[\frakl]$ specified above. We now apply \cref{lem:elkies-iff-frob-split}, using the fact that $\chi_\ell(\sigma_\frakp) \equiv N_{F/\Q}(\frakp) \text{ mod } \ell$.
\end{proof}

 \Cref{prop:elkies-frob-nf} indicates that the \v{C}ebotarev density theorem in $F'/F$ will provide information on how often a fixed prime $\frakl$ is Elkies for the reduced abelian varieties $A_{\frakp}$ as~$\frakp$ grows. In order to apply this theorem, we need to know what the Galois group $\Gal(F'/F)$ is: this is the purpose of the ``large Galois image'' hypothesis in \cref{thm:main-cv}.

\subsection{Large Galois images}
\label{subsec:openimage}

We keep notation from \cref{tab:notations}; here,~$F$ is a number field. To formalize the definition of large Galois images used in the introduction, we introduce the following notation. If~$n$ is an integer, we write
\[
    \Zhat_{\geq n} = \prod_{\ell \text{ prime, } \ell\geq n} \Z_\ell.
\]
The $\ell$-adic Galois representations $\rho_\ell:G_F\to \GSp_{2h}(\O\otimes\Z_\ell)$ can be combined into a global representation
\[
    \rhohat_n: G_F \to \GSp_{2h}(\O\otimes \Zhat_{\geq n})
\]

\begin{definition}
\label{def:large-image}
We say that $A$ has \emph{large Galois image} if for some integer~$n\geq 1$, the image of~$\rhohat_n$
contains $\Sp_{2h}(\O\otimes \Zhat_{\geq n})$. Because the cyclotomic character $\chi_\ell$ is surjective for large enough~$\ell$, an equivalent condition is that for some large enough~$n$,
\[
    \rhohat_n(G_F) = \GSp_{2h}(\O\otimes\Zhat_{\geq n}; (\Zhat_{\geq n})^\times).
\]
\end{definition}

In the main results of this paper, \cref{thm:main-cv,thm:main-moments}, we only consider abelian varieties with RM that have large Galois images. In this subsection, we gather some necessary and sufficient conditions for this to happen.

\begin{prop}
    If~$A$ has large Galois image, then $\End_{\overline{\Q}}(A) = \O$. In particular~$A$ is simple of type {\rm I} in Albert's classification.
\end{prop}

\begin{proof}
    Since we assumed the RM embedding $\O\embed\End(A)$ to be primitive, it is sufficient to prove that $\End_{\overline{\Q}}(A)\otimes \Q = K$. Let~$F'$ be a number field over which all endomorphisms of~$A$ are defined. Since~$G_{F'}$ is an open subgroup of finite index in~$G_F$, there exists a prime~$\ell$ such that $\rho_\ell(G_{F'})$ still contains $\Sp_{2h}(\O\otimes\Z_\ell)$. By Faltings~\cite{faltingsEndlichkeitssaetzeFuerAbelsche1983}, $\End_{F'}(A)\otimes\Q_\ell$ is the commutant of $\rho_\ell(G_{F'})$ in $\End(T_\ell(A)\otimes\Q_\ell)$. 
    
    We claim that the commutant of $\Sp_{2h}(\O\otimes\Z_\ell)$ in $\End(T_\ell(A)\otimes\Q_\ell)$ is precisely given by the action of elements of $\O\otimes\Q_\ell$ on $T_\ell(A)$. This would prove that $\End_{F'}(A)\otimes \Q_\ell$ is contained in $\O\otimes\Q_\ell$, hence $\End_{\overline{\Q}}(A)\otimes\Q = K$ as required.

    To show that the claim holds, choose an element~$\gamma\in \End(T_\ell(A)\otimes\Q_\ell)$ commuting with~$\Sp_{2h}(\O\otimes\Z_\ell)$. In particular, considering scalar matrices in $\Sp_{2h}$, we see that $\gamma$ acts $\O$-linearly: we can therefore consider~$\gamma$ as a $2h\times 2h$ matrix with coefficients in~$\O\otimes\Q_\ell$. Since~$\O\otimes\Q_\ell$ is a product of fields, it is now sufficient to show that for any field~$k$, the commutant of~$\Sp_{2h}(k)$ consists of scalar matrices only.
    
    This last fact is well-known (the Lie algebra representation of $\mathfrak{sp}_{2h}$ on $\mathfrak{sl}_{2h}$ is irreducible), but for completeness,
    we include a short proof when~$k$ is infinite. Let~$\gamma$ be a $2h\times 2h$ matrix over~$k$ commuting with~$\Sp_{2h}(k)$. Consider any symplectic basis $(v_1,\ldots,v_{2h})$ of $k^{2h}$, and let $x_1,\ldots,x_h\in k^\times$ be such that $x_1,\ldots,x_h,x_1^{-1},\ldots,x_h^{-1}$ are distinct. The endomorphism whose matrix in the basis $(v_1,\ldots,v_{2h})$ is $\mathrm{Diag}(x_1,\ldots,x_r,x_1^{-1},\ldots,x_r^{-1})$ is symplectic, so $v_1,\ldots,v_{2h}$ are eigenvectors of~$\gamma$. As each nonzero element of~$k^{2h}$ is part of some symplectic basis, we deduce that each nonzero vector is an eigenvector for~$\gamma$, hence~$\gamma$ is a scalar.
\end{proof}

Conversely, we have the following theorem, after results of Serre~\cite{serreResumeCoursAu1985}, Ribet~\cite{ribetGaloisActionDivision1976}, Chi~\cite{chi$ell$adic$lambda$adicRepresentations1992} and Banaszak--Gajda--Kraso\'n~\cite{banaszakImage$ell$adicGalois2006}.

\begin{theorem}
\label{thm:open-image}
    Assume that $\End_{\overline{\Q}}(A) = \O$ and either:
    \begin{itemize}
        \item $d = 1$ and $g \in \{2, 6\}$, or
        \item $h = g/d$ is odd.
    \end{itemize}
    Then~$A$ has large Galois image.
\end{theorem}

\begin{proof}
    After making a finite extension of~$F$, which only shrinks the image of the Galois representation, we may assume that the Zariski closure $\mathcal{G}_\ell$ of $\rho_\ell(G_F)$ inside $\GSp_{2h}(\Q_\ell)$ is connected for all~$\ell$ \cite[§2.5]{serreResumeCoursAu1985}. After taking another finite extension of~$F$, we may also assume that the $\ell$-adic Galois representations of $A$ are all independent in the sense of \cite[§2.1]{serreResumeCoursAu1985}. The goal is then to prove that $\rho_\ell(G_F)$ contains $\Sp_{2h}(\O\otimes\Z_\ell)$ for large enough~$\ell$. The case $d=1$ is Serre's open image theorem \cite[Thm.~3]{serreResumeCoursAu1985}, while \cite[Thm.~6.16]{banaszakImage$ell$adicGalois2006} covers the cases where $h$ is odd (and can be applied as $\mathcal{G}_\ell$ is connected.)
\end{proof}

In particular, if $\End_{\overline{\Q}}(A) = \O$ and $A$ is either an abelian surface or has odd dimension, then~$A$ has large Galois image.

\section{Counting split matrices in \texorpdfstring{$\GSp_{2h}(\F_q)$}{GSp}}
\label{sec:gsp}

Our goal here is to provide estimates for the cardinality of $\Split{2h}{\F_q}(\lambda_0)$ for $\lambda_0 \in \F_q^\times$. In \cref{sec:distribution}, we will use them with $\F_q = \O/\frakl\O$ when applying the \smash{\v{C}ebotarev} density theorem.

Since $\Split{2h}{\F_q}(\lambda_0)$ is conjugacy-invariant, it is a union of conjugacy classes of $\GSp_{2h}(\F_q)$, and those have been classified: see for instance \cite[Section 6.2]{williams12}. One key element of the classification is the characteristic polynomial, so we start by studying its link with Elkies primes in §\ref{subsec:polElkies}. We use this to count elements in $\Split{2h}{\F_q}(\lambda_0)$, up to negligible terms, in §\ref{subsec:countSplit}.

\subsection{Characteristic polynomials and Elkies primes} \label{subsec:polElkies}

\begin{lemma}
\label{lem:block-triangular}
    Let~$m\in \GSp_{2h}(\F_q)$. Then~$m$ leaves a maximal isotropic subspace of~$\F_q^{2h}$ stable if and only if~$m$ is conjugate in~$\GSp_{2h}(\F_q)$ to a matrix of the form
    \[
    \begin{pmatrix} w & \star \\ 0 & \lambda(m) w^{-\tp} 
    \end{pmatrix}
    \]
    for some $w\in \GL_h(\F_q)$, where $w^{-\tp}$ denotes the inverse transpose of~$w$.
\end{lemma}

\begin{proof}
    Assume that~$m$ stabilizes a maximal isotropic subspace~$V\subset \F_q^{2h}$. Then we can find a symplectic basis of~$\F_q^{2h}$ whose first~$h$ vectors generate $V$, i.e.~we can find $Q\in \Sp_{2h}(\F_q)$ such that
    \[ Q m Q^{-1} =
    \begin{pmatrix} w & \star \\ 0 & w'
    \end{pmatrix}
    \]
    where $w,w'\in \GL_h(\F_q)$. Because~$QmQ^{-1}$ belongs to $\GSp_{2h}(\F_q)$ and~$\lambda(Q m Q^{-1})=\lambda(m)$, we must have $w' = \lambda(m) w^{-\tp}$.
    
    Conversely, assume that~$Q m Q^{-1}$ has the specified form for some $Q\in \GSp_{2h}(\F_q)$. Let~$V$ be the span of the first~$h$ vectors of the canonical basis of~$\F_q^{2h}$. Then~$Q(V)$ is a maximal isotropic subspace of~$\F_q^{2h}$ that is stable under~$m$.
\end{proof}

\begin{definition}
 For a monic polynomial $P \in \F_q[X]$ of degree~$r$ with constant coefficient $a_0 \in \F_q^\times$ and $\lambda_0\in \F_q^\times$, we define the \emph{$\lambda_0$-reciprocal polynomial} of $P$ to be the monic polynomial
    \[\widetilde{P}^{\lambda_0}(X) = \frac{1}{a_0} X^r P\left(\frac{\lambda_0}{X} \right). \]
\end{definition}

\begin{prop} \label{prop:split-implies-PPtilde}
Let~$m\in \GSp_{2h}(\F_q)$, let $\lambda_0 = \lambda(m)$, and let~$\chi_m$ be the characteristic polynomial of~$m$. If~$m$ is split, then there exists $P\in \F_q[X]$ such that $\chi_m = P \widetilde{P}^{\lambda_0}$. 
\end{prop}

\begin{proof}
    We may assume~$m$ is block-triangular as in \cref{lem:block-triangular}. Let~$P$ denote the characteristic polynomial of~$w$. Then the characteristic polynomial of $\lambda(m)w^{-\tp}$ is~$\widetilde{P}^{\lambda_0}$.
\end{proof}

Our next aim is to prove a partial converse to \cref{prop:split-implies-PPtilde} when $\chi_m$ is squarefree.
For a monic polynomial $P(X) = X^n + a_{n-1} X^{n-1} + \ldots + a_0$ of $\F_q[X]$ of degree~$n$, we denote by $c_P$ its companion matrix: $$c_P = \begin{pmatrix}
        0 & \ldots & \ldots & 0 & -a_0 \\
        1 & \ddots &  & \vdots & -a_1 \\
        0 & 1 & \ddots & \vdots & \vdots \\
        \vdots & \ddots & \ddots & 0 & \vdots \\
        0 & \ldots & 0 & 1 & - a_{n-1} \\
    \end{pmatrix}.$$

\begin{prop} \label{prop:parametrizationSeparable}
    Let $\chi \in \F_q[X]$ be a monic squarefree polynomial of degree~$2h$ of the form $\chi(X) = P \widetilde{P}^{\lambda_0}$ with $P \in \F_q[X]$ and $\lambda_0 \in \F_q^\times$. Factor $P = P_1\cdots P_r\in \F_q[X]$ into irreducible polynomials in $\F_q[X]$.
    Then all the elements of $\GSp_{2h}(\F_q)$ whose characteristic polynomial is $\chi$ and multiplier is $\lambda_0$ are conjugate to 
    the matrix
    \[
    \Diag\left(c_{P_1},\ldots,c_{P_r},\lambda_0 c_{P_1}^{-\tp},\ldots,\lambda_0 c_{P_r}^{-\tp}\right)
    = \begin{pmatrix}
    c_{P_1} & & & & & \\
    & \ddots & & & & \\
    & & c_{P_r} & & & \\
    & & & \lambda_0 c_{P_1}^{-\tp} & & \\
    & & & & \ddots & \\
    & & & & & \lambda_0 c_{P_r}^{-\tp} \\ 
\end{pmatrix}.
    \]
    In particular, they form a single conjugacy class in $\GSp_{2h}(\F_q)$.
\end{prop}

\begin{proof}
    Let $m$ be an element of $\GSp_{2h}(\F_q)$ whose characteristic polynomial is $\chi$ and multiplier is $\lambda_0$. Assume that~$m$ is the matrix of an endomorphism $u$ in a symplectic basis~$(e_i)_{1 \leq i \leq 2h}$ of $\F_q^{2h}$. For every $i$, we write $V_i = \ker(P_i(u))$ and $\widetilde{V_i} = \ker(\widetilde{P}_i^{\lambda_0}(u))$, noting that $P_i \ne \widetilde{P_i}^{\lambda_0}$ because $\chi$ is squarefree.
    By adapting directly Lemma~3.1 of \cite{milnorIsometries} in the case where $t \in \GSp_{2h}(\F_q)$, we see that the subspaces $V_i$ and~$\widetilde{V_i}$ are totally isotropic, and that there is an orthogonal decomposition 
    \[
    \bigoplus_{i=1}^r (V_i \oplus \widetilde{V_i}).
    \]
    For every $i$, let $(\alpha_i,\beta_i)$ be a symplectic basis of $V_i\oplus \widetilde{V_i}$; both $\alpha_i$ and $\beta_i$ have length~$\deg(P_i)$. The concatenation $(\alpha_1,\ldots,\alpha_r,\beta_1, \ldots,\beta_r)$ is a symplectic basis of $\F_q^{2h}.$ Calling $Q \in \Sp_{2h}(\F_q)$ the base change matrix from $(e_i)$ to $(\alpha_1,\ldots,\alpha_r,\beta_1,\ldots,\beta_r)$, we have 
    \[
    m = Q \cdot \Diag(m_1,\ldots,m_r,m_1 ',\ldots,m_r ') \cdot Q^{-1}
    \]
    where $m_i,m_i'\in \GL_{\deg(P_i)}(\F_q)$ for all $i$.
    For every $i$, the characteristic polynomial of $m_i$ is $P_i$, so $m_i$ is conjugate to $c_{P_i}$ in $\GL_{\deg(P_i)}(\F_q)$: there is $R_i \in \GL_{\deg(P_i)}(\F_q)$ such that $m_i = R_i c_{P_i} R_i^{-1}$. We define $$R = \Diag\left(R_1,\ldots,R_r,R_1^{-\tp},\ldots,R_r^{-\tp} \right) \in \Sp_{2h}(\F_q)$$ and we have $$ m = Q R  \cdot \Diag\left(c_{P_1},\ldots,c_{P_r},R_1^{\tp} m_1' R_1^{-\tp}, \ldots, R_r^{\tp} m_r R_r^{-\tp} \right) \cdot R^{-1} Q^{-1}.$$ Because $$\Diag\left(c_{P_1},\ldots,c_{P_r},R_1^{\tp} m_1' R_1^{-\tp}, \ldots, R_r^\tp m_r R_r^{-\tp} \right)$$ is in $\GSp_{2h}(\F_q)$ with multiplier $\lambda_0$, we have $R_i^{\tp} m_i ' R_i^{-\tp} = \lambda_0 c_{P_i}^{-\tp} $ for every $i$, so $m$ is conjugate to the block-diagonal matrix specified in the lemma.
\end{proof}

A direct consequence of \cref{prop:parametrizationSeparable}, noting that the block-diagonal matrix specified there is of the form required by \cref{lem:block-triangular}, is that the converse to \cref{prop:split-implies-PPtilde} holds when~$\chi_m$ is squarefree.

In fact, following a referee's kind suggestions, we are able to prove the
stronger result that the converse to \cref{prop:split-implies-PPtilde} holds in
general. This result will not be used in the rest of the paper, so the reader
could directly proceed to~§\ref{subsec:countSplit}, where our counting
arguments rely on \cref{prop:parametrizationSeparable} instead (which does not
hold in the non-squarefree case).

\begin{prop}
\label{prop:elkies-equiv-PPtilde}
    Let~$m\in \GSp_{2h}(\F_q)$, and let~$\lambda_0 = \lambda(m)$. Then $m\in \Split{2h}{\F_q}(\lambda_0)$ if and only if its characteristic polynomial~$\chi_m$ factors as~$\chi_m= P\widetilde{P}^{\lambda_0}$ for some~$P\in \F_q[X]$.
\end{prop}

We start with a lemma.

\begin{lemma}
\label{lem:exists-isotropic-subspace}
    Let~$m\in \GSp_{2h}(\F_q)$, let~$\lambda_0 = \lambda(m)$, and assume that $\chi_m$ is of the form~$R^h$ where~$R$ is monic, irreducible of degree~$2$ and satisfies $\widetilde{R}^{\lambda_0} = R$. Assume that~$h\geq 2$. Then there exists a $2$-dimensional isotropic subspace $V\subset \F_q^{2h}$ stable under~$m$.
\end{lemma}

\begin{proof}
We prove this lemma by induction on~$h$. In any case, considering the Jordan decomposition of~$m$, there always exist a $2$-dimensional subspace $V\subset \F_q^{2h}$ that is stable under~$m$, and on which the characteristic polynomial of~$m$ is~$R$. Because $m\in \GSp_{2h}(\F_q)$, it is easy to see that $m$ also stabilizes the orthogonal complement~$V^\perp$ of~$V$. We may assume that~$V$ is not isotropic, in other words $V\cap V^\perp = \{0\}$. Since the symplectic form is nondegenerate, we have $\dim V + \dim V^\perp = 2h$, which implies that we have an orthogonal decomposition
\begin{displaymath}
    \F_q^{2h} = V \oplus V^\perp.
\end{displaymath}

If~$h > 2$, then by induction there exists a 2-dimensional isotropic subspace $W\subset V^\perp$, so we are done. Assume now $h=2$, and write $R = (X-\mu)(X-\lambda_0/\mu)$ for some $\mu\in \F_{q^2}\setminus\F_q$. Extending scalars to~$\F_{q^2}$, we have a direct sum decomposition
\begin{displaymath}
V = V_1 \oplus W_1 \quad\text{and}\quad V^\perp = V_2\oplus W_2
\end{displaymath}
where the~$V_i$ (resp.~$W_i$) for~$i=1,2$ are generated by eigenvectors of~$m$ for the eigenvalue~$\mu$ (resp.~$\lambda_0/\mu$). Let~$v_1,v_2$ be generators of~$V_1$ and~$V_2$ respectively, and let~$w_1,w_2$ be their Galois conjugates, which generate~$W_1$ and~$W_2$ respectively. Denoting the alternating form by $\langle \cdot,\cdot\rangle$ and the generator of $\Gal(\F_{q^2}/\F_q)$ by $\sigma$, we have
\begin{displaymath}
    \sigma(\langle v_1,w_1\rangle) = \langle w_1, v_1 \rangle = - \langle v_1,w_1\rangle,
\end{displaymath}
and similarly for~$v_2,w_2$. Hence
\begin{displaymath}
    -\frac{\langle v_1,w_1 \rangle}{\langle v_2,w_2\rangle} \in \F_q^\times.
\end{displaymath}
Let~$\alpha\in \F_{q^2}^\times$ be an element with this specific norm, and define~$W$ as the $\F_{q^2}$-span of the vectors~$v_1 + \alpha v_2$ and $w_1 + \overline{\alpha} w_2$, where the bar denotes conjugation. Then we directly check that~$W$ is isotropic and descends to~$\F_q$.
\end{proof}

For an alternative proof of \cref{lem:exists-isotropic-subspace} in the case $h=2$ (at least when~$q$ is an odd prime), we can use the classification of conjugacy classes in~$\GSp_4(\F_q)$ from \cite[Section 6.2]{williams12}. In particular, each class whose characteristic polynomial splits as $P \widetilde{P}^{\lambda_0}$ admits a representative of the form \[
    m = \begin{pmatrix} w & \star \\ 0 & \lambda(m) w^{-\tp} 
    \end{pmatrix}.  \qedhere\]

\begin{proof}[Proof of \cref{prop:elkies-equiv-PPtilde}]
By assumption, $\chi_m$ factors over~$\F_q[X]$ as
\begin{displaymath}
    \chi_m = \prod_{i=1}^{n_1} \bigl(P_i \widetilde{P}_i^{\lambda_0} \bigr)^{a_i} \cdot \prod_{j=1}^{n_2} Q_j^{2b_j}
\end{displaymath}
where the polynomials~$P_i$, $\widetilde{P}_i^{\lambda_0}$ and $Q_j$ are coprime, the exponents $a_i$ and $b_j$ are integers, and $\widetilde{Q}_j^{\lambda_0} = Q_j$ for each~$1\leq j\leq n_2$. We make the following reductions.
\begin{enumerate}
    \item \emph{We can assume~$n_1=0$ and $n_2 = 1$.} To see this, define $V_i$ (resp.~$V_i'$) as the primary subspace for~$m$ relative to~$P_i$ (resp.~$\widetilde{P}_i^{\lambda_0}$) for each~$1\leq i\leq n_1$, and define~$W_j$ as the primary subspace for~$m$ relative to~$Q_j$ for each~$1\leq j\leq n_2$. As in \cite[Lemma 3.1]{milnorIsometries}, we have
    \begin{displaymath}
        \F_q^{2h} = \bigoplus_{i=1}^{n_1} (V_i\oplus V_i') \overset{\perp}\oplus \bigoplus_{j=1}^{n_2} W_j
    \end{displaymath}
    where the two big direct sums are also orthogonal. Moreover, the spaces~$V_i$ and~$V_i'$ are isotropic. If the result holds when~$n_1=0$ and~$n_2=1$, then one can construct a maximal isotropic subspace~$W_j'\subset W_j$ stable under~$m$ for each~$j$. Then, the direct sum
    \begin{displaymath}
        \bigoplus_{i=1}^{n_1} V_i \oplus \bigoplus_{j=1}^{n_2} W_j'
    \end{displaymath}
    is maximal isotropic in~$\F_q^{2h}$ and stable under~$m$.

    From now on, we assume~$\chi_m = Q^{2b}$ where~$Q\in \F_q[X]$ is irreducible and satisfies~$\widetilde{Q}^{\lambda_0} = Q$. The map $\mu\mapsto \lambda_0/\mu$ is an involution of the roots of~$Q$ in an algebraic closure of~$\F_q$.

    \item \emph{We can assume that this involution has no fixed points.} Otherwise, since~$Q$ is irreducible, we actually have~$Q = X \pm \mu_0$ where~$\mu_0\in \F_q^\times$ is a square root of~$\lambda_0$. Considering the Jordan decomposition of~$m$, we only need to prove that any unipotent element in~$\GSp_{2h}(\F_q)$ stabilizes a maximal isotropic subspace. This is a well-known fact that is easy to prove by induction on the dimension.

    From now on, we assume that the involution $\mu\mapsto \lambda_0/\mu$ has no fixed points among the roots of~$Q$. This implies that~$Q$ has even degree $2r$, and that there exists an irreducible polynomial $R\in \F_q[X]$ of degree~$r$ that is the minimal polynomial of~$\mu + \lambda_0/\mu$ over~$\F_q$ for each root~$\mu$ of $Q$.

    \item \emph{We can assume that~$Q$ has degree~$2$.} Indeed, working over~$\F_{q^r}$ (seen as the splitting field of~$R$), we can factor~$Q$ as
    \begin{displaymath}
        Q = \prod_{\sigma\in \mathrm{Gal}(\F_{q^r}/\F_q)} \sigma(Q'),
    \end{displaymath}
    where~$Q' = X^2 - (\mu + \lambda_0/\mu) X + \lambda_0 \in \F_{q^r}[X]$ is irreducible of degree~$2$. Because the polynomials~$\sigma(Q')$ for $\sigma\in \mathrm{Gal}(\F_{q^r}/\F_q)$ are coprime, another application of \cite[Lemma 3.1]{milnorIsometries} yields an orthogonal decomposition
    \begin{displaymath}
        \F_{q^r}^{2h} = \bigoplus_{\sigma\in \mathrm{Gal}(\F_{q^r}/\F_q)} \ker\bigl(\sigma(Q')^{2b}(m)\bigr).
    \end{displaymath}
    If we are able to construct a maximal isotropic subspace~$V'\subset \ker\bigl(\sigma(Q')^{2b}(m)\bigr)$ defined over~$\F_{q^r}$ and stable under~$m$, then the subspace
    \begin{displaymath}
        V = \bigoplus_{\sigma\in \mathrm{Gal}(\F_{q^r}/\F_q)} \sigma(V')
    \end{displaymath}
    descends to~$\F_q$ and yields a maximal isotropic subspace of~$\F_q^{2h}$ stable under~$m$.
\end{enumerate}

After these reductions, we can assume that $h = 2b$ is even, and that $\chi_m = Q^{2b}$ where $Q\in \F_q[X]$ is irreducible of degree~$2$ and satisfies $\widetilde{Q}^{\lambda_0} = Q$. We construct a maximal isotropic subspace $V\subset \F_q^{2h}$ stable under~$m$ by induction. By \cref{lem:exists-isotropic-subspace}, we can construct a $2$-dimensional isotropic subspace $W\subset \F_q^{2h}$ stable under~$m$. If~$h=2$, we are done with~$V=W$. Otherwise, the quotient space $W^\perp/W$ carries an induced alternating form which is nondegenerate, and~$m$ induces an endomorphism of that space. By the induction hypothesis on~$h-2$, there exists a maximal isotropic subspace $V'\subset W^\perp/W$ stable under (the endomorphism induced by)~$m$. We let~$V$ be the preimage of~$V'$ under the quotient map~$W^\perp\to W^\perp/W$, which is maximal isotropic for dimension reasons.
\end{proof}

\subsection{Estimating the size of \texorpdfstring{$\Split{2h}{\F_q}(\lambda_0)$}{S2h}} \label{subsec:countSplit}

We recall that $$\# \GSp_{2h}(\F_q) = (q-1) \cdot \#\Sp_{2h}(\F_q) = (q-1) \cdot q^{h^2} \cdot \prod\limits_{i=1}^{h} (q^{2i}-1) = q^{2h^2 + h+1} + O(q^{2h^2 + h}).$$ In the following, we will write $f(h) = 2h^2 + h +1$. As in \cref{thm:main-cv}, we set
\[
\alpha_h = \sum\limits_{(d_1,\ldots,d_r) \in \Sigma_h} \frac{1}{2^r} \cdot \prod\limits_{i=1}^r \frac{1}{d_i} \cdot \prod\limits_{k=1}^h \frac{1}{\#\{j \; : \; d_j = k \} ! }.
\]
The main result in this subsection is the following. Recall that the notation $O_h$ means that the implied constants are allowed to depend on~$h$, but not on~$\lambda_0$.

\begin{prop} \label{prop:sizeSplit}
    We have $\# \Split{2h}{\F_q}(\lambda_0) = \alpha_h q^{f(h)-1} + O_h\left( q^{f(h)-2} \right).$ 
\end{prop}

Let $\Split{2h}{\F_q}^{\sqf}(\lambda_0)$ (resp. $\Split{2h}{\F_q}^{\nsqf}(\lambda_0)$) be the set of elements $m \in \Split{2h}{\F_q}(\lambda_0)$ such that $\chi_m$ is squarefree (resp.~not squarefree). We obviously have 
\[
\Split{2h}{\F_q}(\lambda_0) = \Split{2h}{\F_q}^{\sqf}(\lambda_0)\sqcup \Split{2h}{\F_q}^{\nsqf}(\lambda_0),
\]
and we will count elements in each subset separately.

\begin{lemma} \label{lem:insepMuFixed}
     We have $
     \# \Split{2h}{\F_q}^{\nsqf}(\lambda_0)  = O_h(q^{f(h)-2}).
     $
\end{lemma}

\begin{proof}
    Define $\GSp_{2h}^{\nsqf}(\F_q;\{\lambda_0\})$ as the set of elements of $\GSp_{2h}(\F_q)$ of multiplier~$\lambda_0$ and whose characteristic polynomial is not squarefree. We will in fact prove the stronger claim
    \[
        \#\GSp_{2h}^{\nsqf}(\F_q;\{\lambda_0\}) = O_h(q^{f(h)-2}).
    \]
    To this end, we wish to view $\GSp_{2h}^{\nsqf}(\F_q;\{\lambda_0\})$ as the set of $\F_q$-points of a certain variety. Let $\Delta : \GSp_{2h} \to \mathbb{A}^1$ be the morphism which maps $m$ to the discriminant of its characteristic polynomial. The points $m \in \GSp_{2h}$ for which $\Delta(m) = 0$ are precisely the elements whose characteristic polynomial is not squarefree. Moreover, the restriction of the morphism $\lambda: \GSp_{2h} \to \Gm$ to elements $m$ for which $\Delta(m) = 0$ is surjective on geometric points (hence as a map of schemes): indeed, if~$\lambda_1$ is a point of~$\Gm$, then~$\lambda(\sqrt{\lambda_1} \id_{2h}) = \lambda_1.$ Thus, the set of points of $\GSp_{2h}$ of multiplier $\lambda_0$ and whose characteristic polynomial is not squarefree is a subvariety of $\GSp_{2h}$ of dimension $\dim(\GSp_{2h})-2 = f(h)-2$, defined by polynomial equations whose degrees are independent of~$\lambda_0$.
    
    In \cite{lang54}, Lang and Weil prove that the number of points defined over $\F_q$ of a variety of dimension $r$ is $O(q^{r})$, where the implicit constant only depends on the dimension and the degree of the variety. We conclude that $\# \GSp_{2h}^{\nsqf}(\F_q;\lambda_0) = O_h(q^{f(h)-2})$.
\end{proof}

We now estimate the size of $\Split{2h}{\F_q}^{\sqf}(\lambda_0).$ 
For a partition $(d_1,\ldots,d_r)$ of the integer $h$ such that $d_1 \leq \ldots \leq d_r$, we denote by $D_{(d_1,\ldots,d_r)}(\lambda_0)$ the set of conjugacy classes contained in $\GSp_{2h}(\F_q)(\lambda_0)$ whose characteristic polynomial is squarefree and factors as 
\[
P_1 \cdots P_r \cdot \widetilde{P}_1^{\lambda_0} \cdots \widetilde{P}_r^{\lambda_0}
\]
where~$P_i\in \F_q[X]$ is irreducible of degree~$d_i$ for every $i \in \{1,\ldots,r \}.$ By \cref{prop:parametrizationSeparable}, the set $D_{(d_1,\ldots,d_r)}(\lambda_0)$ is in one-to-one correspondence with those characteristic polynomials. Moreover, $\Split{2h}{\F_p}^{\sqf}(\lambda_0)$ is the union of all elements of $D_{(d_1,\ldots,d_r)}(\lambda_0)$ as $(d_1,\ldots,d_r)$ runs through partitions of~$h$. 

We first show that the conjugacy classes of $D_{(d_1,\ldots,d_r)}(\lambda_0)$ all have the same size. Recall that for a element $m \in \GSp_{2h}(\F_q)$, the number of elements conjugate to $m$ is $$\#\GSp_{2h}(\F_q)/\#C(m)$$ where $C(m)$ is the centralizer of $m$.

\begin{lemma} \label{CommuteCompanion}
    Let $P(X) = X^n + a_{n-1} X^{n-1} + \ldots + a_0$ be a monic irreducible polynomial of $\F_q[X]$ of degree~$n$. The number of elements $m \in \GL_n(\F_q)$ which commute with the companion matrix $c_P$ is equal to $q^{n} - 1.$ 
\end{lemma}

\begin{proof}
    Let $u$ be the endomorphism of $\F_q^n$ associated to the matrix $c_P$. Then, for every nonzero $x \in \F_q ^n$, the family $(x,u(x),\ldots,u^{n-1}(x))$ is a basis of $\F_q^n$ because $P$ is irreducible. An element $v$ in $C(u)$ is determined by $v(x)$ since for every $i \in \{0,\ldots,n-1\}$, we have $v(u^{i}(x)) = u^{i}(v(x))$. If $v(x) \ne 0$, then $v$ maps the basis $(x,u(x),\ldots,u^{n-1}(x))$ to the basis $(v(x),u(v(x)),\ldots,u^{n-1}(u(x))),$ so $v$ is invertible. Therefore, the elements $m \in \GL_n(\F_q)$ commuting with $c_P$ are in one-to-one correspondence with the nonzero elements of~$\F_q ^n$.
\end{proof}

\begin{lemma} \label{lem:cardinalityConjClass}
    With the above notation, the cardinality of each element of $D_{(d_1,\ldots,d_r)}(\lambda_0)$ is
    \[
    \frac{\#\GSp_{2h}(\F_q)}{(q-1) \prod\limits_{i=1}^r (q^{d_i}-1)}.
    \]
\end{lemma}

\begin{proof}
    By \cref{prop:parametrizationSeparable}, a representative of the class is the block-diagonal matrix
    \[ m = \Diag\left(c_{P_1},\ldots,c_{P_r},\lambda_0 c_{P_1}^{-\tp},\ldots,\lambda_0 c_{P_r}^{-\tp}\right).
    \]
    Matrices in $C(m)$ preserve the invariant subspaces of $m$, so they are also block-diagonal of the form 
    \[
    \Diag \left(
    N_1,\ldots, N_r, \lambda' N_1^{-\tp}, \ldots,\lambda' N_r^{-\tp} \right)
    \]
    where $N_i$ commutes with $c_{P_i}$ for every $i$. By Lemma \ref{CommuteCompanion}, the number of elements $N_i$ in~$\GL_{d_i}(\F_q)$ which commute with $c_{P_i}$ is $q^{d_i}-1.$ Hence,
    \[
    \# C(m) = (q-1) \prod\limits_{i=1}^r (q^{d_i}-1),\]
    where the first factor $(q-1)$ corresponds to the choice of the multiplier $\lambda'$.
\end{proof}

Second, we estimate the size of $D_{(d_1,\ldots,d_r)}(\lambda_0).$ We denote by $ I_{(d_1,\ldots,d_r)}(\lambda_0) $ the set of tuples of irreducible polynomials $(P_1,\ldots,P_r) \in \F_q[X]^r$ such that $P_i$ is irreducible of degree~$d_i$ for all $i$ and the product 
$P_1 \cdots P_r \cdot \widetilde{P}_1^{\lambda_0} \cdots \widetilde{P}_r^{\lambda_0}$ 
is squarefree. In the next lemma, we identify $D_{(d_1,\ldots,d_r)}(\lambda_0)$ with a set of characteristic polynomials.

\begin{lemma} \label{lem:ItoD}
    Consider the map
    \[
    H:\left\{
    \begin{array}{lll}
       I_{(d_1,\ldots,d_r)}(\lambda_0)  & \to & D_{(d_1,\ldots,d_r)}(\lambda_0) \\
        (P_1,\ldots,P_r)  & \mapsto & P_1 \cdots P_r \cdot \widetilde{P}_1^{\lambda_0} \cdots \widetilde{P}_r^{\lambda_0}. \\
    \end{array}
    \right.\]
    Then, for every $\chi \in D_{(d_1,\ldots,d_r)}(\lambda_0)$, we have $$\#H^{-1}(\chi) = 2^r \cdot \prod\limits_{k=1}^{h} \#\{ j \; : \; d_j = k \}!.$$
\end{lemma}

\begin{proof}
    Fix an element $(P_1,\ldots,P_r) \in H^{-1}(\chi)$. Then, because~$\chi$ is squarefree, choosing another element of $H^{-1}(\chi)$ consists in choosing one element of the pair $\{P_i,\widetilde{P}_i^{\lambda_0} \}$ for every $i \in \{1,\ldots,r\}$, as well as a permutation of the tuple $(P_{i_{k,1}},\ldots,P_{i_{k,s}})$ where $P_{i_{k,1}},\ldots,P_{i_{k,s}}$ are the polynomials of degree~$k$, for every $k \in \{1,\ldots,h\}$.
\end{proof}

\begin{proof}[Proof of \cref{prop:sizeSplit}]
By \cref{lem:cardinalityConjClass}, the number of elements in $\Split{2h}{\F_q}^{\sqf}(\lambda_0)$ is $$
    \#\Split{2h}{\F_q}^{\sqf}(\lambda_0) = \sum\limits_{(d_1,\ldots,d_r) \in \Sigma_h} \#D_{(d_1,\ldots,d_r)}(\lambda_0) \cdot \frac{\#\GSp_{2h}(\F_q)}{(q-1) \prod\limits_{i=1}^r (q^{d_i}-1)}.$$
    
For a partition $(d_1,\ldots,d_r)$ of $h$ with $d_1 \leq \ldots \leq d_r$, we estimate the size of $D_{(d_1,\ldots,d_r)}(\lambda_0)$ by determining the size of $I_{(d_1,\ldots,d_r)}(\lambda_0)$ and using \cref{lem:ItoD}.

The last coefficients of a monic polynomial $P$ of degree~$d$ such that $P = \widetilde{P}^{\lambda_0}$ are determined by the first coefficients, so the number of irreducible polynomials $P$ such that $P = \widetilde{P}^{\lambda_0}$ is $O(q^{d-1}).$ For every $i$, one has to choose $P_i$ such that $P_i \ne \widetilde{P}_i^{\lambda_0}$ and $P_i \ne P_j, \widetilde{P}_j^{\lambda_0}$ for indices $j < i$. Then, according to a formula from Gauss for the number of irreducible polynomials of $\F_q[X]$ of given degree (see \cite{chebolu11} for a proof), the number of choices for~$P_i$ is $\frac{1}{d_i} q^{d_i} + O(q^{d_i -1}).$ Hence
\[
\#I_{(d_1,\ldots,d_r)}(\lambda_0) = \left(\prod\limits_{i=1}^{r} \frac{1}{d_i} \right) q^{h} + O_h(q^{h-1}).
\]
Therefore,
\begin{align*}
    \#D_{(d_1,\ldots,d_r)}(\lambda_0) & = \left( \frac{1}{2^r} \cdot\prod\limits_{i=1}^r  \frac{1}{d_i} \cdot \prod\limits_{k=1}^h  \frac{1}{\# \{ j \; : \; d_j = k \}!} \right) q^{h} + O_h(q^{h-1}) 
\end{align*}
and consequently
\[
\#\Split{2h}{\F_q}^{\sqf}(\lambda_0) = \alpha_h q^{f(h)-1} + O(q^{f(h)-2}).
\]
Combining this with Lemma \ref{lem:insepMuFixed} ends the proof.
\end{proof}

Even using \cref{prop:elkies-equiv-PPtilde}, we note that giving an exact count of $\Split{2h}{\F_q}^{\nsqf}(\lambda_0)$ would be more difficult than in the squarefree case, because a given characteristic polynomial may correspond to several conjugacy classes, contrary to \cref{prop:parametrizationSeparable}. For our purposes, the asymptotic upper bound of \cref{lem:insepMuFixed} is sufficient.

As a final remark, when $h=2$ and $q = \ell$ is an odd prime, we are able to determine the exact cardinality of $\Split{2h}{\F_\ell}$, through the classification of the conjugacy classes of $\GSp_4(\F_\ell)$ and the computation of the cardinality of the classes in \cite{breeding11}.

\begin{prop}
\label{prop:exact-count}
    We have $$\# \Split{4}{\F_\ell} =\frac{(3\ell^3 + 7\ell^2 + 7\ell + 11)(\ell + 1)(\ell - 1)^3\ell^4}{8}.$$
\end{prop}

\begin{proof}
 Conjugacy classes of $\GSp_4(\F_\ell)$ have been sorted in different types \cite[Section 6.2]{williams12} according to the factorization of the characteristic polynomial, and the number of classes of each type is known.  
    The order of each center has been computed in \cite[Table 1]{breeding11} (beware that the antisymmetric matrix used to define $\GSp_{2h}$ there is not $J_{2h}$, so notation differs from \cite{williams12}). Thus we can deduce the size of each conjugacy class; we omit the detailed calculation.
\end{proof}

\section{The distribution of Elkies primes}
\label{sec:distribution}

In this section, we prove Theorem \ref{thm:main-moments}. We introduce the character sum $U_k$, similar to the sum~$U$ in \cite[eq.\,(4)]{shparlinski15}, in §\ref{subsec:setup}. We control the small terms in this sum in §\ref{subsec:smallTerms} and we estimate the dominant term in §\ref{subsec:dominantTerm}. Finally, we conclude the proof in §\ref{subsec:conclusion}.

\subsection{Setup} \label{subsec:setup}

We keep the notation from \cref{thm:main-cv}. We may assume that $P$ and $L$ are sufficiently large, so that $A_\frakp$ is well-defined for every $\frakp \in \mathcal{P}_F(P,2P)$, and if $\L=\frakl_1 \cdots \frakl_r$ is the product of $r$ distinct primes of $\mathcal{P}_K(L,2L)$, then 
$$G_\L := \Gal(F(A[\L])/F)$$
contains $\Sp_{2h}\left( \O / \L \O \right)$. (Recall that we always have $G_\L\subset \GSp_{2h}\left(\O / \L \O \right)$.) This assumption is harmless since we want to establish an asymptotic result.

The Landau prime ideal theorem \cite{landau} for the fields $K$ and $F$ asserts that 
\[
\#\mathcal{P}_K(L,2L) \sim \frac{L}{\log(L)} \quad \mathrm{and} \quad \#\mathcal{P}_F(P,2P) \sim \frac{P}{\log(P)}.
\]
Let 
\[
\delta_{\frakp,\frakl} =
    \begin{cases}
        (1-\alpha_h) & \mbox{if $\frakl$ is Elkies for $A_\frakp$}, \\
        -\alpha_h & \mbox{otherwise}
    \end{cases}
\]
For a product $ \frakl_1 \cdots \frakl_r$, we define 
\[
\delta_{\frakp,\frakl_1\cdots \frakl_r} = \delta_{\frakp,\frakl_1} \cdots \delta_{\frakp,\frakl_r}.
\]
We further set
\[
\mu    = \alpha_h  \# \mathcal{P}_K(L,2L) \quad\text{and}\quad
        \sigma  = \sqrt{\alpha_h  (1-\alpha_h)  \# \mathcal{P}_K(L,2L)}.
\]
By definition,
\begin{align*}
N_e(\frakp,L) - \mu & = (1-\alpha_h) N_e(\frakp,L) - \alpha_h (\# \mathcal{P}_K(L,2L)-N_e(\frakp,L)) 
\\ & = \sum\limits_{\frakl \in \mathcal{P}_K(L,2L)} \delta_{\frakp,\frakl}.
\end{align*} 
For any integer $k \geq 1$, the $k$-th moment of $X_{P,L}$ is 
\begin{align*}
    \mathbb{E}(X_{P,L}^k) 
    & = \frac{1}{\# \mathcal{P}_F(P,2P)}   \sum\limits_{\frakp \in \mathcal{P}_F(P,2P) } \left(\frac{N_e(\frakp,L) - \mu}{\sigma} \right)^k \\
    & = \frac{1}{\# \mathcal{P}_F(P,2P) \cdot \sigma^k}  \sum\limits_{\frakp \in \mathcal{P}_F(P,2P)} \left( \sum\limits_{\frakl \in \mathcal{P}_K(L,2L)} \delta_{\frakp,\frakl} \right)^k \\
    & = \frac{1}{\# \mathcal{P}_F(P,2P) \cdot \sigma^k}  \sum\limits_{\frakp \in \mathcal{P}_F(P,2P)}  \sum\limits_{\substack{\frakl_1,\ldots,\frakl_k \\ \in \mathcal{P}_K(L,2L)}} \delta_{\frakp,\frakl_1 \cdots \frakl_k}.
\end{align*} 
Hence, we are led to considering the sums 
\[
U_k:= \sum\limits_{\frakp \in \mathcal{P}_F(P,2P)} \sum\limits_{\substack{\frakl_1,\ldots,\frakl_k \\ \in \mathcal{P}_K(L,2L)}} \delta_{\frakp,\frakl_1 \cdots \frakl_k}.
\]

We expect compensations in the sum $U_k$ when some primes among $\frakl_1,\ldots,\frakl_k$ appear an odd number of times, and we will sort terms according to the number of distinct primes. In the spirit of the proof of \cite[Theorem 1]{shparlinski15}, for $0\leq j \leq k$, let  $\mathcal{Q}_{k,j}$ be the set of tuples $(\frakl_1,\ldots,\frakl_k)$ of primes in $\mathcal{P}_K(L,2L)$ such that $\frakl_1 \cdots \frakl_k = \fraka^2 \frakb$ where $\frakb$ is a squarefree product of $j$ prime ideals and $\fraka$ is the product of $\frac{k-j}{2}$ prime ideals ($\mathcal{Q}_{k,j}$ is empty if $k-j$ is odd). If $k = 2\nu$ is even, we also define $\mathcal{Q}_{k,0}'\subset \mathcal{Q}_{k,0}$ to be the set of tuples $(\frakl_1,\ldots,\frakl_k)$ such that $\frakl_1 \cdots \frakl_k = \fraka ^2$ where $\fraka$ is a product of $\nu$ \emph{distinct} prime ideals. We will see that the dominant term comes from the contribution of the terms of $\mathcal{Q}_{k,0}'$. We begin by estimating the other terms.

\subsection{Small terms} \label{subsec:smallTerms}

We want to prove the following result, which generalizes Lemmas~5 and~6 in \cite{shparlinski15}. As in \cref{thm:main-cv}, the dependency on~$A$ in Landau's notation includes the dependency on~$F$, $\O$ and~$h$.

\begin{prop} \label{prop:sumLiFixed}
Assume GRH. For $P > 2L$ and a product $\L = \frakl_1 \ldots \frakl_r$  of $r$ distinct primes of $\mathcal{P}_K(L,2L)$, we have 
\[
\sum\limits_{\frakp \in \mathcal{P}_F(P,2P)} \delta_{\frakp,\L} = O_{A,r}\left(\frac{P}{ \log(P) L^r} + L^{f(h)r} P^{1/2} \log(P)  \right).
\]
\end{prop}

The proof of this proposition is based on the \v{C}ebotarev density theorem in the Galois group $G_\L,$ which is a subgroup of
\[
\GSp_{2h}\left( \prod\limits_{i=1}^r \O/\frakl_i \O \right) \cong \prod\limits_{i=1}^r \GSp_{2h}(\O/\frakl_i\O).
\]
Thus, an element $m \in G_\L$ can be identified with an element 
\[
(m_1,\ldots,m_r) \in \prod\limits_{i=1}^r \GSp_{2h}(\O/\frakl_i \O)
\]
and its multiplier $\lambda(m)$ with an element 
\[
(\lambda_1,\ldots,\lambda_r) \in \prod\limits_{i=1}^r (\O/\frakl_i\O)^\times.
\]
For $(\lambda_1,\ldots,\lambda_r) \in \lambda(G_\L),$ we define
\[
G_{\L}(\lambda_1,\ldots,\lambda_r) := \{ m \in G_{\L}:  \lambda(m) = (\lambda_1,\ldots,\lambda_r) \}
\]
which can be identified with $$\prod\limits_{i=1}^r \GSp_{2h}(\O/\frakl_i\O,\{\lambda_i\}).$$ In particular, by the large Galois image assumption,
\[\# G_{\L}(\lambda_1,\ldots,\lambda_r) = \prod\limits_{i=1}^r \# \Sp_{2h}(\O/\frakl_i \O)
\quad \mathrm{and} \quad \# G_{\L} = \# \lambda(G_{\L}) \cdot \prod\limits_{i=1}^r \# \Sp_{2h}(\O/\frakl_i \O).
\]

Let us now construct the conjugacy-invariant subsets of~$G_\L$ we are interested in. Given a tuple $\eps = (\eps_1,\ldots,\eps_r) \in \{\pm 1\}^r$, we denote by 
\[
\mathcal{C}_{\frakl_1,\ldots,\frakl_r}(\eps_1,\ldots,\eps_r) \subset G_\L
\]
the set of elements $m = (m_1,\ldots,m_r)$ of $G_{\L}$ such that $m_i \in \Split{2h}{\O/\frakl_i\O}$ if $\eps_i = 1$ and $m_i \notin \Split{2h}{\O/\frakl_i \O}$ if $\eps_i = -1$. These sets~$\mathcal{C}_{\frakl_1,\ldots,\frakl_r}(\eps_1,\ldots,\eps_r)$ are indeed stable by conjugation in~$G_\L$, and form a partition of~$G_\L$ as~$\eps$ varies.

We also need to determine the size of these sets.
For $\lambda_i \in (\O/\frakl_i \O)^\times$, we define 
\[
    \begin{cases}    
    C_{\frakl_i}^1(\lambda_i) & = \# \Split{2h}{\O/\frakl_i\O}(\lambda_i), \\
        C_{\frakl_i}^{-1}(\lambda_i) & = \# (\GSp_{2h}(\O/\frakl_i \O;\{\lambda_i\}) - \Split{2h}{\O/\frakl_i\O}(\lambda_i)).
\end{cases}
\]
Considering the preimage of each multiplier $(\lambda_1,\ldots,\lambda_r) \in \lambda(G_{\L})$ separately, we immediately obtain
\[
\# \mathcal{C}_{\frakl_1, \ldots,\frakl_r }(\eps_1,\ldots,\eps_r) = \sum\limits_{(\lambda_1,\ldots,\lambda_r) \in \lambda(G_{\L})} \prod\limits_{i=1}^r C_{\frakl_i}^{\eps_i}(\lambda_i).
\]

For a given prime~$\frakp$ of good reduction for~$A$, if we set 
\[
\eps_{\frakp,\frakl_i} = \begin{cases}
       1 & \mbox{if $\frakl_i$ is Elkies for $A_\frakp$}, \\
       -1 & \mbox{otherwise}.
    \end{cases}
\]
then by \cref{prop:elkies-frob-nf}, the Frobenius element $\sigma_\frakp$ at~$\frakp$ in $G_\L$ satisfies 
\[
(\rhobar_{\frakl_1}(\sigma_\frakp),\ldots,\rhobar_{\frakl_r}(\sigma_\frakp))
\in \mathcal{C}_{\frakl_1,\ldots,\frakl_r}(\eps_{p,\frakl_1},\ldots,\eps_{p,\frakl_r}).
\]

\begin{lemma}
\label{lem:cebotarev}
    Let $\L = \frakl_1\cdots\frakl_r$ be a product of distinct primes of $\mathcal{P}_K(L,2L)$ and $(\eps_1,\ldots,\eps_r)$ be an element of $\{\pm 1\}^r$. Then, for $x > 2L$, we have
    \begin{align*}
        \#\{ \frakp \; : \; N(\frakp) \leq x \; \text{and} \; \eps_{\frakp,\frakl_i} = \eps_i \text{ for all } i\} 
        & = \frac{\sum\limits_{(\lambda_1,\ldots,\lambda_r) \in \lambda(G_{\L})} \prod\limits_{i=1}^r C_{\frakl_i}^{\eps_i}(\lambda_i)}{\# G_{\L}} \frac{x}{\log(x)} \\
        & \quad + O_{A,r}\left(N(\frakl_1)^{f(h)} \cdots N(\frakl_r)^{f(h)}   x^{1/2} \log(x)\right). \\  
        \end{align*}
\end{lemma}

\begin{proof}
    This follows from an effective version of the \v{C}ebotarev density theorem in~$G_\L$ \cite[§2, Equation ($20_R$)]{serre81} for the set $\mathcal{C}_{\frakl_1,\ldots,\frakl_r}(\eps_1,\ldots,\eps_r)$.
    In the left-hand side of ($20_R$), an upper bound on $\#\mathcal{C}_{\frakl_1, \ldots,\frakl_r }(\eps_1,\ldots,\eps_r)$ is the order of $G_\L$, which is $O_r\left(N(\frakl_1)^{f(h)} \cdots N(\frakl_r)^{f(h)}\right).$ The degree $n$ of the extension $F(A[\L])/F$ is equal to the order of $G_\L$. Since $x > 2L$, we have $\log(n) = O_{r,h}(\log(x)).$ For every $i \in \{1, \ldots,r\}$, let $\ell_i$ be the prime number below~$\frakl_i$. The ramified primes in the extension $F(A[\L])/F$ lie among the divisors of $\ell_1, \ldots, \ell_r$ and the primes of bad reduction of $A$ (this follows from the Néron-Ogg-Shafarevich criterion), so 
    \[
    \log\left(\prod\limits_{i=1}^r \ell_i \right) = O_r(\log(x))
    \]
    under the assumption $x > 2L$.
\end{proof}

\begin{proof}[Proof of \cref{prop:sumLiFixed}]
    We have 
    \[
    \sum\limits_{\frakp \in \mathcal{P}_F(P,2P)} \delta_{\frakp,\L} = \sum\limits_{\substack{(\gamma_1,\ldots,\gamma_r) \\ \in \{ 1-\alpha_h,-\alpha_h\}^r }} \gamma_1 \cdots \gamma_r \cdot \#\{ \frakp : N(\frakp) \leq x \; \text{and} \; \eps_{\frakp,\frakl_i} = \eps_i \text{ for all } i\}
    \]
    where $\eps_i = 1$ if $\gamma_i = 1 - \alpha_h$ and $\eps_i = -1$ if $\gamma_i = - \alpha_h.$ Write 
    \[
    S_{(\lambda_1,\ldots,\lambda_r)}(\frakl_1,\ldots,\frakl_r) = \sum\limits_{\substack{(\gamma_1,\ldots,\gamma_r) \\ \in \{1-\alpha_h,-\alpha_h \}^r }} \gamma_1 \cdots \gamma_r \cdot \prod\limits_{i=1}^r C_{\frakl_i}^{\eps_i}(\lambda_i).
    \]
    By the previous lemma,
    \begin{equation}
        \label{eq:sum-delta}
    \begin{aligned}
        \sum\limits_{\frakp \in \mathcal{P}_F(P,2P)} \delta_{\frakp,\L} 
        & = \frac{\sum\limits_{(\lambda_1,\ldots,\lambda_r) \in \lambda(G_{\L})} S_{(\lambda_1,\ldots,\lambda_r)}(\frakl_1, \ldots, \frakl_r)}{\# G_{\L}} \cdot \frac{P}{\log(P)} \\
        & \quad + O_{A,r}\left(N(\frakl_1) ^{f(h)} \cdots N(\frakl_r) ^{f(h)} P^{1/2} \log(P)\right). \\
    \end{aligned}
    \end{equation}
    Fix $(\lambda_1, \ldots,\lambda_r) \in \lambda(G_{\L})$. Then, 
    \[
    S_{(\lambda_1,\ldots,\lambda_r)}(\frakl_1,\ldots,\frakl_r) = \prod\limits_{i=1}^r S_{\lambda_i}(\frakl_i)
    \]
    where $S_{\lambda_i}(\frakl_i) := (1-\alpha_h) C_{\frakl_i}^1(\lambda_i) - \alpha_h C_{\frakl_i}^{-1}(\lambda_i).$
    We have $S_{\lambda_i}(\frakl_i) = O_h(N(\frakl_i)^{f(h)-2})$ according to \cref{prop:sizeSplit}. Thus, 
    \[
    S_{(\lambda_1,\ldots,\lambda_r)}(\frakl_1,\ldots,\frakl_r) = O_{A,r}\left(N(\frakl_1)^{f(h)-2} \cdots N(\frakl_r)^{f(h)-2}\right).
    \]
    In this equation, the implicit constant is independent of $(\lambda_1,\ldots,\lambda_r).$ Therefore, 
    \[
    \sum\limits_{(\lambda_1,\ldots,\lambda_r) \in \lambda(G_{\L})} S_{(\lambda_1,\ldots,\lambda_r)}(\frakl_1, \ldots, \frakl_r) = O_{A,r} \left( \# \lambda(G_{\L}) \cdot N(\frakl_1)^{f(h)-2} \cdots N(\frakl_r)^{f(h)-2}\right).
    \]
    Consequently, 
    \[\frac{\sum\limits_{(\lambda_1,\ldots,\lambda_r) \in \lambda(G_{\L})} S_{(\lambda_1,\ldots,\lambda_r)}(\frakl_1, \ldots, \frakl_r)}{\# G_{\L}} = O_{A,r} \left(\frac{1}{L^r}\right).\]
    Inserting this upper bound and writing $N(\frakl_i)\leq 2L$ in~\eqref{eq:sum-delta} ends the proof.
\end{proof}

\subsection{The dominant term} \label{subsec:dominantTerm}
When $k$ is even, the dominant term of $U_k$ corresponds to the contribution of elements of $\mathcal{Q}_{k,0}'$. We begin by estimating the size of this set. For a positive integer $\nu$, we recall that $M_{2\nu}$ is the moment of order $2 \nu$ of the standard Gaussian distribution. Its value is $$M_{2\nu} = (2\nu-1) \cdot (2\nu-3) \cdots 3 \cdot 1.$$

\begin{lemma} \label{countingQ0}
    Let $\nu$ be a positive integer. Then,
    \begin{align*}\# \mathcal{Q}_{2\nu,0}' 
    & = M_{2\nu} \frac{L^{\nu}}{\log(L)^{\nu}} + O_\nu \left( \frac{L^{\nu-1}}{\log(L)^{\nu-1}} \right), \quad \text{and}\\
    \#(\mathcal{Q}_{2\nu,0} - \mathcal{Q}_{2\nu,0}')
    & =  O_\nu \left( \frac{L^{\nu-1}}{\log(L)^{\nu-1}} \right).
    \end{align*}
\end{lemma}

\begin{proof}

For $n \in \{1,\ldots,\nu\}$, let $\mathcal{A}_n$ be the set of tuples $(A_1,\ldots,A_n)$ of disjoint subsets of $\{1,\ldots,2\nu \}$ such that:
\begin{itemize}
    \item for every $i \in \{1,\ldots,n\}$, $A_i \ne \emptyset,$
    \item for every $ i \in \{1,\ldots,n\}$, $\# A_i$ is even,
    \item $\bigsqcup\limits_{i=1}^n A_i = \{1,\ldots,2\nu\}.$
\end{itemize}
We equip $\mathcal{P}_K(L,2L)$ with an arbitrary total order $<$. We also define $\mathcal{B}_L^n$ to be the set of ordered $n$-tuples of distinct prime ideals of $\mathcal{P}_K(L,2L)$. Let $s = (\frakl_1,\ldots,\frakl_{2\nu})$ be an element of $\mathcal{Q}_{2\nu,0}$ such that $\lcm(\frakl_1,\ldots,\frakl_{2\nu})$ has $n$ distinct prime factors, and $\frakl_1 ' < \ldots < \frakl_n '$ be the primes such that $$\{\frakl_1,...\frakl_{2\nu}\} = \{ \frakl_1',\ldots,\frakl_n ' \}.$$  Then, we define $b_s = (\frakl_1',\ldots,\frakl_n')$. For $j \in \{1,\ldots,n\}$, we set $$A_j ^s = \{i \in \{1,\ldots,2\nu\} \; : \; \frakl_i = \frakl_j' \}$$ and $a_s = (A_1 ^s,\ldots,A_n ^s).$ 

With this notation, the set $\mathcal{Q}_{2\nu,0}$ is in one-to-one correspondence with $$\bigsqcup_{1\leq n \leq \nu} \mathcal{A}_n \times \mathcal{B}_L ^n$$ via $s \mapsto (a_s,b_s),$ and $\mathcal{Q}_{2\nu,0}'$ is in one-to-one correspondence with $\mathcal{A}_\nu \times \mathcal{B}_L ^\nu.$  
 
If $n$ is fixed, we have 
\[
 \#\mathcal{B}_L ^n = \binom{\# \mathcal{P}_K(L,2L)}{n} \sim \frac{L^n}{n!\log(L)^n}
 \]
 as $L$ goes to infinity. For $n = \nu$, we have 
 \[
 \#\mathcal{A}_{\nu} = \binom{2 \nu}{2} \cdot \binom{2\nu-2}{2} \cdots \binom{2}{2} =\nu ! \cdot M_{2\nu},
 \]
 so $\#(\mathcal{A}_{\nu} \times \mathcal{B}_L ^{\nu}) \sim M_{2\nu} \frac{L^{\nu}}{\log(L)^{\nu}}.$ On the other hand, for $n\leq \nu-1$, we have 
 \[
 \# \mathcal{B}_{L}^n = O \left( \frac{L^{\nu-1}}{\log(L)^{\nu-1}} \right).
 \]
 Since $\#\mathcal{A}_{n}$ is a constant independent of $L$, we obtain
 \[
 \sum\limits_{n=1}^{\nu-1} \# \mathcal{A}_n \cdot \# \mathcal{B}_L ^n = O_\nu \left( \frac{L^{\nu-1}}{\log(L)^{\nu-1}} \right). \qedhere
 \] 
\end{proof}

We are able to prove a more precise statement than \cref{prop:sumLiFixed} for elements of $\mathcal{Q}_{2\nu,0}'.$

\begin{prop} \label{prop:dominantSum}
    Let $(\frakl_1,\ldots,\frakl_{2\nu}) \in \mathcal{Q}_{2\nu,0}'$. Then, 
    $$\sum\limits_{\frakp \in \mathcal{P}_F(P,2P)} \delta_{\frakp,\frakl_1 \cdots \frakl_{2\nu}} = (\alpha_h (1-\alpha_h))^{\nu} \frac{P}{\log(P)} + O_{A,\nu}\left( \frac{P}{\log(P)L} + L^{f(h)\nu} P^{1/2} \log(P)\right).$$
\end{prop}

\begin{proof}
    Assume that $\{ \frakl_1 , \ldots , \frakl_{2 \nu} \} = \{ \frakl_1 ' , \ldots, \frakl_{\nu}'\}$ where $\frakl_1' < \ldots < \frakl_{\nu}'.$ Given $(\gamma_1,\ldots,\gamma_{\nu})$ in $ \{ 1-\alpha_h,-\alpha_h\}^\nu$, denote by $\mathcal{D}_{\frakl_1',\ldots, \frakl_{\nu}'}(\gamma_1,\ldots,\gamma_\nu)$ the set of primes $\frakp \in \mathcal{P}_F(P,2P)$ such that for every $i$, the prime $\frakl_i '$ is Elkies for $A_\frakp$ if $\gamma_i = 1-\alpha_h $, and $\frakl_i ' $ is not Elkies for $A_\frakp$ if $\gamma_i = -\alpha_h.$ As in \cref{lem:cebotarev}, the \v{C}ebotarev density theorem yields:
\begin{align*}
    \#\mathcal{D}_{\frakl_1',\ldots, \frakl_{\nu}'}(\gamma_1,\ldots,\gamma_\nu) & = \frac{\sum\limits_{(\lambda_1, \ldots, \lambda_\nu) \in \lambda(G_{\frakl_1' \cdots \frakl_\nu '})} \prod\limits_{i=1}^\nu C_{\frakl_i '}^{\eps_i}(\lambda_i) }{\#G_{\frakl_1' \cdots \frakl_\nu'}} \frac{P}{\log(P)} + O_{A,\nu}\left( L^{\nu f(h)} P^{1/2} \log(P) \right) \\
    & = (\alpha_h)^{k_1} (1-\alpha_h)^{\nu-k_1} \frac{P}{\log(P)} + O_{A,\nu}\left(\frac{P}{\log(P) L} + L^{\nu f(h)} P^{1/2} \log(P) \right) \\ \end{align*}
    where $k_1$ is the number of entries $\gamma_i$ equal to $1-\alpha_h.$ Then,
    \begin{align*}
     \sum\limits_{\frakp \in \mathcal{P}_F(P,2P)} \delta_{\frakp,\frakl_1 \cdots \frakl_{2\nu}} 
     &= \sum\limits_{\substack{(\gamma_1,\ldots, \gamma_\nu) \\ \in \{ 1-\alpha_h,-\alpha_h \}^\nu }} (1-\alpha_h)^{2 k_1} (\alpha_h)^{2(\nu-k_1)} \cdot \# D_{\frakl_1',\ldots, \frakl_{\nu}'}(\gamma_1,\ldots,\gamma_\nu) \\
      &= \sum\limits_{k_1=0}^{\nu} \binom{\nu}{k_1} \left( 1-\alpha_h \right)^{2k_1} \left( \alpha_h \right)^{2(\nu-k_1)} \left( \alpha_h \right)^{k_1} \left( 1-\alpha_h \right)^{\nu-k_1} \frac{P}{\log(P)} \\ 
     &\qquad + O_{A,\nu}\left( \frac{P}{\log(P) \cdot L} + L^{f(h)\nu} P^{1/2} \log(P) \right) \\
      &= (\alpha_h  (1-\alpha_h))^\nu \frac{P}{\log(P)} + O_{A,\nu}\left( \frac{P}{\log(P) \cdot L} + L^{f(h)\nu} P^{1/2} \log(P)\right). \qedhere
\end{align*}
\end{proof}

\subsection{Conclusion of the proof} \label{subsec:conclusion}

We go back to estimating the moments of~$X_{P,L}$.
First, assume that $k$ is odd, and write $k= 2 \nu +1$. Then
\[
U_k = \sum\limits_{j=0}^{\nu} \sum\limits_{ \substack{(\frakl_1,\ldots,\frakl_{2 \nu +1}) \\ \in \mathcal{Q}_{2\nu+1,2j+1}}} \sum\limits_{\frakp \in \mathcal{P}_F(P,2P)} \delta_{\frakp,\frakl_1 \ldots \frakl_{2 \nu+1}}.
\]
For $j\in \{0,\ldots,\nu\}$, we have 
\[
\#\mathcal{Q}_{2\nu+1,2j+1} = O_\nu \left( \frac{L^{\nu+j+1}}{\log(L)^{\nu+j+1}} \right),
\]
so by Proposition \ref{prop:sumLiFixed},  
\begin{align*}
&\sum\limits_{\substack{(\frakl_1,\ldots,\frakl_{2 \nu +1}) \\ \in \mathcal{Q}_{2\nu+1,2j+1}}} \sum\limits_{\frakp \in \mathcal{P}_F(P,2P)} \delta_{\frakp, \frakl_1 \cdots \frakl_{2\nu+1}} \\
&\qquad = O_{A,\nu}\left( \frac{L^{\nu+j+1}}{\log(L)^{\nu+j+1}} \left( \frac{P}{L^{2j+1}\log(P)} + L^{f(h)(2j+1)} P^{1/2} \log(P) \right) \right).
\end{align*}
The dominant terms occur for $j=0$ and $j=\nu$. By getting rid of the non-dominant terms, we obtain 
\[
U_{k} = O_{A,k}\left( \frac{L^{\nu}P}{\log(L)^{\nu+1} \log(P)} + \frac{L^{(2 \nu +1) (f(h)+1)}P^{1/2} \log(P)}{\log(L)^{2\nu+1}} \right).
\]

We finally plug this upper bound into the expression for $\mathbb{E}(X_{P,L}^k)$ in~§\ref{subsec:setup}. We have
\[
\sigma^k \cdot \# \mathcal{P}_F(P,2P) \underset{P,L \to + \infty }{\sim} (\alpha_h (1-\alpha_h))^{\nu+1/2}\frac{P}{\log(P)} \cdot \frac{L^{\nu+1/2}}{\log(L)^{\nu + 1/2}}.\]
Hence, \begin{align*} 
\mathbb{E}(X_{P,L}^k) & = \frac{U_k}{\# \mathcal{P}_F(P,2P) \sigma^k} \\
& = O_{A,k} \left( \frac{1}{L^{1/2}\log(L)^{1/2} } + \frac{L^{(2\nu+1)(f(h)+1)-\nu-1/2}  \log(P)^2}{\log(L)^{\nu +1/2} P^{1/2} } \right),
\end{align*}
proving \cref{thm:main-moments} for odd~$k$.

Second, assume that $k = 2 \nu$ is even. We also write 
\[
U_{2 \nu } = \sum\limits_{j=0}^{\nu} \sum\limits_{(\frakl_1,\ldots,\frakl_{2 \nu}) \in \mathcal{Q}_{2\nu,2j}} \sum\limits_{\frakp \in \mathcal{P}_F(P,2P)} \delta_{\frakp,\frakl_1 \cdots \frakl_{2\nu}}.
\]
For $j \in \{1,\ldots,\nu\}$, we obtain as above
\[
\sum\limits_{(\frakl_1,\ldots,\frakl_{2 \nu }) \in \mathcal{Q}_{2\nu,2j}} \sum\limits_{\frakp \in \mathcal{P}_F(P,2P)} \delta_{\frakp,\frakl_1 \cdots \frakl_{2\nu}} = O_{A,\nu}\left( \frac{L^{\nu+j}}{\log(L)^{\nu+j}} \left( \frac{P}{L^{2j}\log(P)} + L^{2f(h)j} P^{1/2} \log(P) \right) \right).
\]

Now assume that $j=0$. By \cref{countingQ0} and \cref{prop:dominantSum}, the contribution of elements of $\mathcal{Q}_{2\nu,0}'$ to $U_k$ is
\begin{align*} & M_{2\nu} \frac{L^{\nu}}{\log(L)^{\nu}}(\alpha_h \cdot (1-\alpha_h))^{\nu} \frac{P}{\log(P)} \\ 
& \quad + O_{A,\nu}\left( \frac{P}{\log(P) L} + \frac{P L^{\nu-1}}{\log(P) \log(L)^{\nu-1}} + \frac{L^{(f(h)+1)\nu} P^{1/2} \log(P) }{\log(L)^\nu}  \right) 
\end{align*}
while the contribution of elements from $\mathcal{Q}_{2\nu,0} - \mathcal{Q}_{2\nu,0}'$ is
\[
O_\nu \left( \frac{P \log(L)^{\nu-1}}{\log(P) \log(L)^{\nu-1}} \right).
\]

The dominant terms in the above upper bounds occur for $j=0$ and $j=\nu$, and we have 
\begin{align*}
    U_{2\nu}  & = M_{2\nu}\frac{L^{\nu}}{\log(L)^{\nu}}\left( \alpha_h \cdot (1-\alpha_h) \right)^{\nu} \frac{P}{\log(P)}\\ & \quad + O_{A,\nu} \left( \frac{P}{\log(P)} \frac{L^{\nu-1}}{\log(L)^{\nu-1}} + \frac{L^{(2f(h)+2)\nu} P^{1/2} \log(P)}{\log(L)^{2\nu}} \right).
\end{align*} 
Therefore,
\begin{align*}
    \mathbb{E}(X_{P,L}^k) = M_k + O_{A,k}\left( \frac{\log(L)}{L} + \frac{L^{(2f(h)+1)\nu} \log(P)^2 }{\log(L)^{\nu} P^{1/2} } \right).
\end{align*}
This concludes the proof of \cref{thm:main-moments}; \cref{thm:main-cv} is a consequence of this theorem and \cite[Theorem 30.2]{billingsley95}.

\section{Numerical experiments}
\label{sec:experiments}

At the beginning of this project, we performed numerical experiments with \textit{SageMath} \cite{sagemath} in order to confirm experimentally the estimate of \cite[Theorem 1]{shparlinski15}. All the experiments presented in this section were made with the non-CM elliptic curve $E$ given by the Weierstrass equation $y^2 + y = x^3 - x^2$ defined over $\Q$ (Cremona label 11a3). We began by computing some values of the left-hand side of \cite[Theorem 1]{shparlinski15} for $\nu =1$, namely
$$\frac{1}{\pi(2P)-\pi(P)} \sum\limits_{p \in \Prime_\Q(P,2P)} \left( N_e(p,L) - \frac{\pi(2L)-\pi(L)}{2} \right)^2,$$
by fixing one of the variable $L$ or $P$ and by letting the other one vary. \cref{fig:momentLFixed} shows the evolution of the left-hand side for three values of $L$ (namely 25, 100 and 250) and $P$ varying between $10^3$ and $5 \cdot 10^6$.

\begin{figure}[ht]
    \centering
    \includegraphics[width=0.7\linewidth]{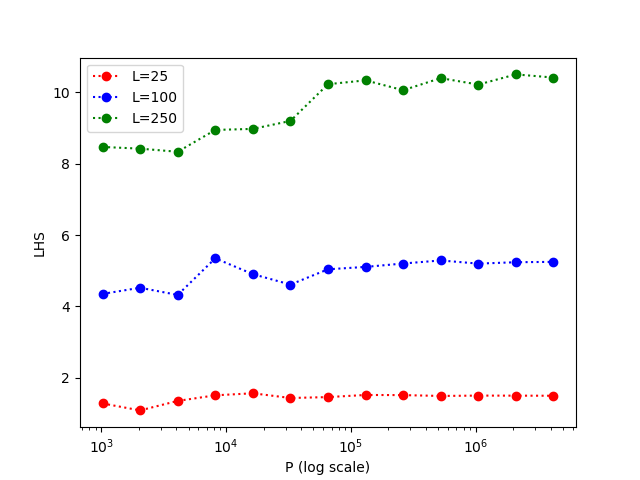}
    \caption{Moment of order $2$ for $P \in [10^3, 5 \cdot 10^6]$}
    \label{fig:momentLFixed}
\end{figure}

This graph suggests that the left-hand side has a finite limit (which depends on $L$) as~$P$ goes to infinity. To go further, we analyzed the distribution of~$N_e(p,L)$, i.e.~the number of primes $p \in \mathcal{P}_\Q(P,2P)$ such that $N_e(p,L) = n$ as $n$ varies between $0$ and $\pi(2L)-\pi(L)+1.$ We observed that this distribution has a Gaussian shape when $P$ is much larger than $L$ as in \cref{fig:distribution}.

We then tried to predict the mean value and the standard deviation as a function of $L$ through a naive probabilistic model, relying on the the standard hypothesis that roughly $50\%$ of prime numbers are Elkies. (Indeed, for a given elliptic curve $E$ defined $\F_q$ of trace of Frobenius $t$, the prime $\ell$ is Elkies if and only if $t^2-4q$ is a square modulo $\ell$, and half of the elements of $\F_\ell^\times$ are squares; we neglect the probability that $t^2-4q$ is 0 modulo~$\ell$.) In other words, for every $p \in \mathcal{P}_\Q(P,2P)$, a prime $\ell \in [L,2L]$ has a probability $1/2$ to be Elkies for the reduced curve~$E_p$, and those events are independent. Then, the number of Elkies primes for $E_p$ in $[L,2L]$ follows a binomial distribution $\mathcal{B}(\pi(2L)-\pi(L),1/2)$, whose expected value~$\mu$ and deviation~$\sigma$ are
\[
\mu = \frac{\pi(2L)-\pi(L)}{2} \quad\text{and}\quad \sigma = \frac{\sqrt{\pi(2L)-\pi(L)}}{2}.
\]
Therefore, when $P$ is much larger than $L$, we expect the actual distribution of Elkies primes to look like a Gaussian function with those parameters. In \cref{fig:distribution}, we plot the distribution for $L=250$ and $P=10^7$ in blue and the associated Gaussian red; we see that the naive model fits very well with the reality.

\begin{figure}[ht]
    \centering
    \includegraphics[width=0.7\linewidth]{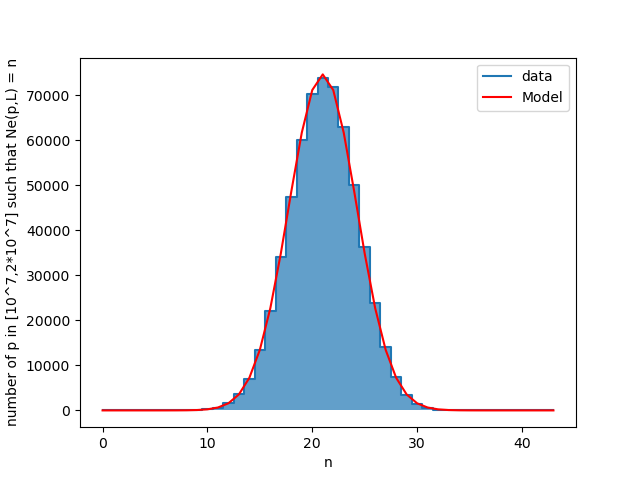}
    \caption{Distribution with $L=250$ and $P=10^7$}
    \label{fig:distribution}
\end{figure}

The predicted value of the left-hand side of \cite[Theorem 1]{shparlinski15} for $\nu=1$ is the moment of order $2$ of the binomial distribution $\mathcal{B}(\pi(2L)-\pi(L),1/2)$, which is $(\pi(2L)-\pi(L))/4$. In \cref{fig:evolutionMoment}, we fix $P = 10^5$ and we let $L$ vary in $[20,500]$. We plot the evolution of the left-hand side for in blue and the predicted value in red. We see that the model is accurate for small values of $L$, but when $L$ is larger than $\sqrt{P}$, a gap between the model and the reality starts appearing.

\begin{figure}[ht]
    \centering
    \includegraphics[width=0.7\linewidth]{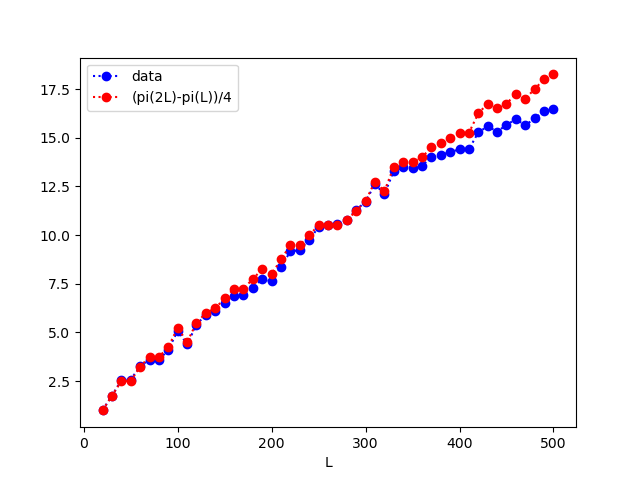}
    \caption{Evolution of the moment of order 2, with $P = 10^5$ and $L \in [20,500]$}
    \label{fig:evolutionMoment}
\end{figure}

All in all, these numerical experiments gave us the idea that the distribution of Elkies primes converges to a Gaussian function when $P$ and $L$ go to infinity with $P$ growing quickly compared with $L$. The naive model allowed us to predict the parameters of this Gaussian function in the setting of elliptic curves, setting us on the path towards \cref{thm:main-cv}.


\end{document}